\documentclass[a4paper]{amsart}

\usepackage{amsfonts,amssymb,amsmath,amsthm,latexsym,bm}
\usepackage{url}
\usepackage{enumitem}
\usepackage{mathtools}
\usepackage{multirow}

\setlist[enumerate]{font=\normalfont}
\setlist[enumerate,1]{label=(\roman*)}

\newtheorem{theorem}{Theorem}[section]

\newtheorem{example}[theorem]{Example}
\newtheorem{lemma}[theorem]{Lemma}
\newtheorem{proposition}[theorem]{Proposition}
\newtheorem{corollary}[theorem]{Corollary}
\theoremstyle{definition}
\newtheorem{definition}[theorem]{Definition}
\theoremstyle{remark}
\newtheorem{remark}[theorem]{Remark}
\newtheorem{problem}[theorem]{Problem}
\newcommand{\cc}{{\mathbb C}}
\newcommand{\qq}{{\mathbb Q}}
\newcommand{\rr}{{\mathbb R}}
\newcommand{\nn}{{\mathbb N}}
\newcommand{\zz}{{\mathbb Z}}

\newcommand{\calP}{\mathcal{P}}

\newcommand{\abs}[1]{\left|{#1}\right|}

\DeclareMathOperator{\rank}{rank}

\author{Teruyuki~Mishima}
\address{Accenture Japan Ltd \\
	2-2-2 Nakanoshima, Kita-ku, Osaka 530-0005 \\
	Japan}
\email{teruyuki.mishima@accenture.com}

\author{Xiao-Nan~Lu}
\address{
	Department of Electrical, Electronics and Computer Engineering,
	Faculty of Engineering, Gifu University \\
	1-1 Yanagido, Gifu-shi, Gifu 501-1193 \\
	Japan}
\email{lu.xiaonan.b3@f.gifu-u.ac.jp}

\author{Masanori~Sawa}
\address{Graduate School of System Informatics, Kobe University \\
	1-1 Rokkodai, Nada, Kobe, Hyogo 657-8501 \\
	Japan}
\email{sawa@people.kobe-u.ac.jp}

\author{Yukihiro~Uchida}
\address{Department of Mathematical Sciences, Graduate School of Science, Tokyo Metropolitan University \\
         1-1 Minami-Osawa, Hachioji, Tokyo 192-0397 \\
         Japan}
\email{yuchida@tmu.ac.jp}

\thanks{This research is supported in part by Grant-in-Aid for Early-Career Scientists 22K13949 and Grant-in-Aid for Scientific Research (C) 20K03517 of the Japan Society for the Promotion of Science (JSPS), and by the Early Support Program for Grant-in-Aid for Scientific Research of Kobe University.}



\begin{document}

\title[Geometric designs and Hilbert-Kamke equations]{Geometric designs and Hilbert-Kamke equations of degree five for classical orthogonal polynomials}

\begin{abstract}
In this paper we elucidate the advantage of examining the connections between Hilbert-Kamke equations and geometric designs, or Chebyshev-type quadrature, for classical orthogonal polynomials.
We first establish that if a $5$-design with $6$ rational points for a symmetric classical measure is parametrized by rational functions, then the corresponding measure should be the Chebyshev measure $(1-t^2)^{-1/2}dt/\pi$ on $(-1,1)$. Our proof is based on the collaboration of a certain polynomial identity and some advanced techniques on the computation of the genus of a certain irreducible curve.
Next, we prove a necessary and sufficient condition for the existence of rational $5$-designs for the Chebyshev measure.
Moreover, as one of our main theorems, we construct an infinite family of ideal solutions for the Prouhet-Tarry-Escott (PTE) problem by utilizing rational $5$-designs for the Chebyshev measure, and then establish that, up to affine equivalence over $\mathbb{Q}$, such ideal solutions are included in the famous parametric solutions found by Borwein (2002).
\end{abstract}

\maketitle

\section{Introduction}
\label{sect:Intro}

Let $w(t)$ be a probability density function on an interval $I \subseteq \rr$, with finite moments
\begin{equation}\label{eq:moment}
    a_k = \int_I t^k w(t) dt, \quad k= 0,1,\dotsc
\end{equation}
An integration formula of type
\begin{equation}\label{eq:QF}
	\frac{1}{n} \sum_{i=1}^n f(x_i) 	= \int_I f(t) w(t) dt,
\end{equation}
where $x_1,\ldots,x_n \in I$, is a \emph{Chebyshev-type quadrature of degree $m$} if the equation (\ref{eq:QF}) holds for every $f \in \mathcal{P}_m$.
We denote by $\mathcal{P}_m$ the vector space of all univariate polynomials of degree at most $m$. 

As implied by various formulas named after Gauss, the construction of quadrature formula has long been studied in mathematical analysis,
in connection with the investigation of the zeros of orthogonal polynomials. For a good introduction to the interrelation between quadrature and orthogonal polynomials, we refer the reader to Dunkl and Xu~\cite[\S~1]{DX2014} and Szeg\H{o}~\cite[\S~15]{Sze1975}.
On the other hand, the configuration of points $x_1,\ldots,x_n$, especially for Gegenbauer measure $(1-t^2)^{\lambda-1/2}dt/\int_{-1}^1 (1-t^2)^{\lambda-1/2}dt$ where $\lambda > -1/2$, has captured the attention of researchers in algebraic combinatorics and related areas.
The point configuration is called an \emph{$m$-design} if all $x_i$ are distinct, and a \emph{weighted $m$-design} otherwise.
A good survey paper on algebro-combinatorial aspects of quadrature is Bannai and Bannai~\cite{BB2009}.

Most of the designs appearing in the present paper are \emph{rational}, namely the point configuration is included in $\qq$.
The notion of rational design originally goes back to Hausdorff~\cite{Hau09}, who established the existence of a weighted rational design for Hermite measure $e^{-t^2}dt/\sqrt{\pi}$ on $(-\infty,\infty)$ and thereby simplified Hilbert's solution of the Waring problem~\cite{Hi09}.
Interest was revived by Nesterenko~\cite[p.4700]{N06}, where a refinement of Hausdorff's method is discussed; see also Subsection~\ref{subsect:Hausdorff} of this paper. 
It is noteworthy that the construction of rational designs can be applied in the realization of solutions of the Waring problem.

The construction of rational designs is also significant in algebraic combinatorics, in particular in the theory of spherical designs.
A challenging problem, asked by Bannai, Bannai, Ito and Tanaka~\cite[p.208, Problem 2]{BBI2016}, is to present an explicit construction of rational designs for Gegenbauer measure $(1-t^2)^{\lambda-1/2}dt/\int_{-1}^1 (1-t^2)^{\lambda-1/2}dt$.
Such constructions, if realized, can be applied in the explicit construction of spherical designs~\cite{B1992,RB1991,W1991}.
What is meant by  \emph{explicit construction} is that it generates a design for which the components of each point are written down explicitly.
It is still an open problem to present an explicit construction of rational designs for Gegenbauer measure in general; see for example~\cite[pp.347--348]{X2022}. An important aim of this paper is to make substantial progress on this topic.

On the other hand, a rational weighted $m$-design of type \eqref{eq:QF} is equivalent to a solution of a system of Diophantine equations of type
\begin{equation}\label{eq:HK0}
     \begin{split}
        x_1 + x_2 + \dotsb + x_n &= n a_1, \\
        x_1^2 + x_2^2 + \dotsb + x_n^2 &= n a_2, \\
        & \vdotswithin{=} \\
        x_1^m + x_2^m + \dotsb + x_n^m &= n a_m.
    \end{split}
\end{equation}
In particular an $m$-design is equivalent to a \emph{disjoint solution} of (\ref{eq:HK0}), meaning that all $x_i$ are mutually distinct.
We call (\ref{eq:HK0}) \emph{Hilbert-Kamke equations}, following Cui, Xia and Xiang~\cite{CXX2019}.
A systematic treatment of the equations (\ref{eq:HK0}) was established by Sawa and Uchida~\cite{SU2020} for probability measures that correspond to the classical orthogonal polynomials. It is well known (see~\cite[Subsection 2.1]{SU2020}) that such measures are completely classified by Hermite, Laguerre and Jacobi measures.
The work of Sawa and Uchida~\cite{SU2020} was extended to Bessel measure over $\cc$ by Matsumura~\cite{M23}.

We are here mainly concerned with \emph{symmetric classical measures},
that is Hermite or Gegenbauer measure, for which the odd moments $a_{2k+1}$ are identically zero.
An $m$-design is \emph{antipodal} if the configuration of points $x_1,\ldots,x_n$ is symmetric with respect to the origin $0$.
An antipodal $m$-design with $2N$ points for symmetric classical measures is equivalent to the equations
\begin{equation}\label{eq:symHK}
    x_1^{2k} + x_2^{2k} + \dotsb + x_N^{2k} = N a_{2k}, \quad k=1,2,\ldots,r,
\end{equation}
where $r$ is the integer part of $m/2$.
A main purpose of this paper is to make clear the speciality of rational designs or Hilbert-Kamke equations for symmetric classical measures, especially for the Chebyshev measure $\pi^{-1} (1-t^2)^{-1/2}dt$ on $(-1,1)$.

This paper is also motivated by a recent work by Misawa et al.~\cite{MMS2024}, where conditions for various geometric designs to be antipodal are discussed. Their work includes an interesting observation that if $m \equiv 1 \pmod{2}$ and $n \le m+1$, then interval $m$-designs with $n$ points for symmetric classical measures have only antipodal solutions; for more details, see Subsection~\ref{subsect:geometric} of this paper. Inspired by their work, we restrict our attention to antipodal designs of odd degrees for symmetric classical measures.

The smallest nontrivial case to be considered is the antipodal $3$-design. In this case, we  prove a necessary and sufficient condition for the existence of antipodal $3$-designs for the Hermite measure.

\begin{theorem}\label{thm:deg3}
There exists an antipodal $3$-design with $n$ rational points for the Hermite measure $e^{-t^2}dt/\sqrt{\pi}$ on $(-\infty,\infty)$ if and only if $n \notin \{1,2,3,7\}$.
\end{theorem}

An analogue of Theorem~\ref{thm:deg3} is also discussed for the Gegenbauer measures.

The next case to be handled is when $m=5$.
The degree five or four case is of special interest in the connection among Hilbert-Kamke equations, geometric designs and the Waring problem; see Remark~\ref{rem:KawadaWooley2} and Proposition~\ref{prop:KawadaWooley} of this paper.
The following is a main theorem of this paper.

\begin{theorem}\label{thm:parametrization0}
If a $5$-design with $6$ rational points for symmetric classical measures is parametrized by rational functions, then $3a_2^2=2a_4$. Moreover, the only symmetric classical measure with $3a_2^2=2a_4$ is the Chebyshev measure $(1-t^2)^{-1/2}dt/\pi$ on $(-1,1)$.
\end{theorem}

Our proof is based on the calculation of the genus of a certain irreducible curve.
More precisely, the equations (\ref{eq:symHK}) define a curve that is not necessarily irreducible.
An irreducible curve is parametrized by rational functions if and only if the curve has genus $0$; for example, see \cite[(2.16)]{Reid}.
Therefore, the proof of Theorem~\ref{thm:parametrization0} reduces to the calculation of the genus of each irreducible component of the curve defined by (\ref{eq:symHK}).
Such calculations of genera were also done in the work of Sawa and Uchida~\cite{SU2020} to examine Hilbert-Kamke equations (\ref{eq:HK0}).
However, we require a different algebro-geometric method in this paper since the curve defined by (\ref{eq:symHK}) is a space curve that is not necessarily irreducible.
It is another aim of this paper to introduce this method to geometric design theory.

Kawada and Wooley~\cite{KW1999} clarified the validity of a certain polynomial identity in the study of the Waring problem for biquadrates (see (\ref{eq:KawadaWooley})).
Theorem~\ref{thm:parametrization0} not only provides theoretical evidence for the difficulty of the existence of rational designs for Gegenbauer measures in general, but also shows a distinctive feature of the Chebyshev measure from the viewpoints of algebraic geometry and combinatorics.

The following two results for the Chebyshev measure are one of the main theorems in this paper.
Their proofs are based on the collaboration of techniques in geometric design theory, elementary number theory, and the same polynomial identity as that used in the proof of Theorem~\ref{thm:parametrization0}.

\begin{theorem}\label{thm:fivesix}
The $5$-designs with $6$ rational points for Chebyshev measure $(1-t^2)^{-1/2}dt/\pi$ on $(-1,1)$ are parametrized by rational functions
\begin{equation}\label{eq:fivesix1_0}
\begin{gathered}
	x_1 =  \frac{2t^2-22t-13}{14(t^2+t+1)}, \;
	x_2 =  \frac{-13t^2-4t+11}{14(t^2+t+1)}, \;
	x_3 =  \frac{11t^2+26t+2}{14(t^2+t+1)}, \\
     x_{i+3} = -x_i \quad \text{for $i=1,2,3$},
\end{gathered}
\end{equation}
where $t \in \qq$.
\end{theorem}

\begin{theorem}\label{thm:spectrum0}
There exists an antipodal $5$-design with $2N$ rational points for Chebyshev measure $(1-t^2)^{-1/2}dt/\pi$ on $(-1,1)$ if and only if
\[
    \begin{split}
        N=3k \ \text{ with } \ k \ge 1, \\
        N=3k+1 \ \text{ with } \ k \ge 6, \\
        N=3k+2 \ \text{ with } \ k \ge 3.
    \end{split}
\]
\end{theorem}

Now, in this paper, we also consider a different type of Diophantine problem called the \emph{Prouhet-Tarry-Escott (PTE) problem}.
Borwein~\cite[Chapter 21]{B2002} presents a brief survey on this subject, including some parametric solutions of the PTE problem together with various individual examples.
We also refer the reader to Dickson's book~\cite[\S~24]{D1920}.

\begin{problem}[Prouhet-Tarry-Escott problem]
\label{prob:PTE1}
For positive integers $M$ and $N$, find a pair of disjoint multisets of integers, say
$X = \{ x_1, \ldots, x_N \}$ and $Y = \{ y_1, \ldots, y_N \}$, such that
\begin{equation}
\label{eq:PTE1}
    \begin{split}
        x_1 + \cdots + x_N &= y_1  + \cdots + y_N, \\
        x_1^2 + \cdots + x_N^2 &= y_1^2 + \cdots + y_N^2, \\
         & \vdots\hspace{25mm} \\
        x_1^M + \cdots + x_N^M &= y_1^M + \cdots + y_N^M.
    \end{split}
\end{equation}
Such a solution is said to have  \emph{degree $M$} and \emph{size $N$}, and is denoted by $[X] =_M^N [Y]$.
\end{problem}

It is well known (cf.~\cite[Chapter 21]{B2002}) that if $[X] =_M^N [Y]$, then $N \ge M+1$; a solution with equality is \emph{ideal}.
The existence of ideal solutions is a central core topic in the study of the PTE problem.
The following degree-five solution due to Borwein is one of the few known ideal parametric solutions:
\begin{equation}\label{eq:Borwein2}
\begin{gathered}
	x_1 = 2 n + 2 m, \quad
	x_2 = -n m - n - m + 3, \quad
	x_3 = n m - n - m - 3, \\
	y_1 = 2 n - 2 m, \quad
	y_2 = - n + n m + m + 3, \quad
	y_3 = m - n m - n - 3, \\
    x_{3+i} = -x_i, \; y_{3+i} = -y_i, \quad i = 1,2,3,
\end{gathered}
\end{equation}
where $m,n \in \mathbb{Z}$.

The following result (Theorem~\ref{thm:PTEparametric}) creates a novel connection between rational interval designs and the PTE problem, where ideal solutions of the PTE problem are obtained from the $5$-designs (\ref{eq:fivesix1_0}). A two-dimensional extension of Theorem~\ref{thm:PTEparametric} or the Borwein solution has been recently discussed by Matsumura et al.~\cite{MS2025}.

\begin{theorem} \label{thm:PTEparametric}
Let
\begin{equation}
\label{eq:PTE2}
X = \{ \pm a_{s,t}b_{u,v}, \pm a_{s,t}c_{u,v}, \pm a_{s,t}d_{u,v} \}, \;
Y = \{ \pm a_{u,v}b_{s,t}, \pm a_{u,v}c_{s,t}, \pm a_{u,v}d_{s,t} \},
\end{equation}
where
\begin{equation}
\label{eq:PTE3}
\begin{gathered}
a_{\alpha,\beta} = \alpha^2 + \alpha \beta + \beta^2, 
b_{\alpha,\beta} = 2\alpha^2 - 22\alpha \beta - 13\beta^2, 
c_{\alpha,\beta} = -13\alpha^2-4\alpha \beta +11 \beta^2, \\
d_{\alpha,\beta} = 11\alpha^2 + 26\alpha \beta + 2 \beta^2.
\end{gathered}
\end{equation}
Then $(X,Y)$ is an ideal solution of the equations (\ref{eq:PTE1}) for $(M,N) = (5,6)$.
Moreover, up to affine equivalence over $\qq$, the solution $(X,Y)$ is included in the Borwein solution (\ref{eq:Borwein2}).
\end{theorem}

This paper is organized as follows.
Section~\ref{sect:preliminaries} gives preliminaries where we review basic facts and background on Hilbert-Kamke equations and geometric designs.
Sections~\ref{sect:small_deg} through~\ref{sect:PTE} are the main body of this paper, where Theorems~\ref{thm:deg3} through~\ref{thm:PTEparametric} are proved, respectively. Section~\ref{sect:conclusion} is the conclusion, where further remarks will be made.

\section{Preliminaries}\label{sect:preliminaries}

In this section we review some basic terminologies and facts that often appear throughout this paper. The description below is fairly conscious of the relationship between Hilbert-Kamke equations, quadrature formulas and Waring problem.
Subsections~\ref{subsect:Hausdorff} and~\ref{subsect:geometric} include some observations that are related to the body of this paper.
In Subsection~\ref{subsect:rational}, we review some terminology on algebraic and arithmetic geometry used in Subsection~\ref{subsect:geometric} and Section~\ref{sect:parametrization}.

\subsection{Quadrature and Hilbert-Kamke equations}\label{subsect:QF}

Let $I \subseteq \rr$ be an interval, with a probability density function $w(t)$ such that all moments $a_k = \int_I t^k w(t)dt$ are finite.

\begin{definition}\label{def:QF1}
An integration formula of type
\begin{equation}\label{eq:QF1}
	\frac{1}{n} \sum_{i=1}^n f(x_i) 	= \int_I f(t) w(t) dt
\end{equation}
where $x_1,\ldots,x_n \in I$, is a \emph{Chebyshev-type quadrature of degree $m$} if (\ref{eq:QF1}) holds for every polynomial $f \in \mathcal{P}_m$.
Here $\mathcal{P}_m$ denotes the space of all univariate polynomials of degree at most $m$. The configuration of points $x_i$ is an \emph{(equi-weighted) $m$-design} if all $x_i$ are distinct, and a \emph{weighted $m$-design} otherwise.
As a convention, we write $(x_1,\ldots,x_n)$ or $\{x_1,\ldots,x_n\}$ for point configurations.
\end{definition}

Most designs and quadrature, often discussed in the present paper, have the following two properties.

\begin{definition}
\hangindent\leftmargini
\label{def:antipodal}
\textup{(i)} A (weighted) $m$-design for $\int_I w(t) dt$ is \emph{antipodal} if the point  configuration $\{x_i\}$ coincides with $\{-x_i\}$.
\begin{enumerate}
\setcounter{enumi}{1}
\item[(ii)] A (weighted) $m$-design for $\int_I w(t) dt$ is \emph{rational} if all $x_i$ are rational numbers.
\end{enumerate}
\end{definition}

Let us look at some examples.

\begin{example}[Newton-Cotes formula]\label{ex:Newton}
The Newton-Cotes quadrature is given as
\begin{equation}\label{eq:NewtonCotes}
	\frac{1}{6} f(-1) + \frac{2}{3} f(0) + \frac{1}{6} f(1)
	= \frac{1}{2} \int_{-1}^1
	f(t) dt \quad \text{for every $f \in \mathcal{P}_3$}.
\end{equation}
The point configuration $\{0,\pm1\}$ is antipodal and rational, but not equi-weighted.
\end{example}

\begin{example}[Chebyshev-Gauss quadrature]\label{ex:Gauss}
Let $x_1,\ldots,x_n$ be the zeros of the Chebyshev polynomial $T_n(t)$ of the first kind.
The Chebyshev-Gauss quadrature is given as
\begin{equation}\label{eq:GaussChebyshev}
	\frac{1}{n} \sum_{i=1}^n f\Big( \cos\Big( \frac{2i-1}{2n} \pi \Big) \Big)
	= \int_{-1}^1 \frac{f(t) dt}{\pi\sqrt{1-t^2}} \quad \text{for every $f \in \mathcal{P}_{2n-1}$}.
\end{equation}
This point configuration is antipodal and equi-weighted, but irrational for $n \ge 2$. 
Among all quadrature of degree $m$ with $n$ points,
the formula (\ref{eq:GaussChebyshev}) is minimal  with respect to the classical lower bound $n \ge \lfloor m/2 \rfloor + 1$; for example, see~\cite{DX2014}.
\end{example}

\begin{example}\label{ex:Chebyshev1}
\begin{equation}\label{eq:Chebyshev1}
\frac{1}{6} \sum_{\pm} \Big\{ f\Big(\pm \frac{1}{7}\Big) + f\Big(\pm\frac{11}{14}\Big) + f\Big(\pm \frac{13}{14} \Big) \Big\} =  \int_{-1}^1 \frac{f(t) dt}{\pi\sqrt{1-t^2}} \quad \text{for every $f \in \mathcal{P}_5$}.
\end{equation}
This point configuration, which appears many times in the present paper,  is antipodal, equi-weighted, and rational. 
The quadrature (\ref{eq:Chebyshev1}) is also minimal with respect to the inequality $N \ge M+1$ for the PTE problem; for the detail, see Section~\ref{sect:PTE}.
\end{example}
 
The following provides a method for composing small designs to produce a larger one of the same degree.

\begin{proposition}\label{prop:composition}
If $X_1$ and $X_2$ are disjoint $t$-designs for $\int_I w(t)dt$, then so is $X_1 \cup X_2$. 
\end{proposition}
\begin{proof}[Proof of Proposition~\ref{prop:composition}]
For $i=1,2$, let $N_i$ be the size of $X_i$. Then it follows, from the defining property (\ref{eq:QF1}) of quadrature, that for every polynomial $f$ of degree at most $t$,
\[
\begin{gathered}[b]
\frac{1}{N_1+N_2} \sum_{x \in X_1\cup X_2} f(x)
= \frac{N_1}{N_1+N_2} \frac{1}{N_1} \sum_{x \in X_1} f(x) + \frac{N_2}{N_1+N_2} \frac{1}{N_2} \sum_{x \in X_2} f(x) \\
= \frac{N_1}{N_1+N_2} \int_I f(t) w(t)dt + \frac{N_2}{N_1+N_2} \int_I f(t) w(t)dt
= \int_I f(t) w(t)dt.
\end{gathered}
\qedhere
\]
\end{proof}

A rational weighted design is equivalent to a certain system of Diophantine equations that originally stems from the Waring problem.

\begin{proposition}[\cite{SU2020}]\label{prop:HK_1}
Let $w(t)dt$ be a probability measure such that all moments $a_k = \int_I t^k w(t)dt$ are finite and rational.
Let $x_1,\ldots,x_n \in \qq \cap I$. Then the following are equivalent:
\begin{enumerate}
\item[(i)] The point configuration $(x_1,\ldots,x_n)$ is a rational (weighted) $m$-design for $\int_I w(t)dt$;
\item[(ii)] $(x_1,\ldots,x_n)$ forms a rational solution of the Diophantine equations
\begin{equation}\label{eq:HK}
  x_1^k + x_2^k + \dotsb + x_n^k = n a_k, \quad k=0,1,\ldots,m. 
\end{equation}
\end{enumerate}
\end{proposition}

The equations (\ref{eq:HK}) were originally designed by Hausdorff~\cite{Hau09} as a simplification of Hilbert's solution of the Waring problem~\cite{Hi09}.
Sawa and Uchida~\cite[Theorem 5.5]{SU2020} generalized weighted designs for Hermite measure to those for all probability measures that correspond to the classical orthogonal polynomials, namely Hermite, Laguerre and Jacobi polynomials.
As a follow-up to their work, Matsumura~\cite{M23} discussed weighted designs for Bessel measure $-e^{-2/z}dz/(4\pi\sqrt{-1})$ over $\mathbb{C}$.

The following question was posed by Bannai, Bannai, Ito and Tanaka~\cite[p.208, Problem 2]{BBI2016}, (see also~\cite{X2022}).

\begin{problem}\label{prob:HK}
Given positive integers $m,n \ge 2$, develop an explicit construction of points $x_1,\ldots,x_n \in \qq \cap (-1,1)$ such that
\begin{equation}\label{eq:HK_equiweighted}
\frac{1}{n} \sum_{i=1}^n f(x_i) = \int_{-1}^1 f(t) (1-t^2)^{(d-3)/2} dt \quad
\text{for every $f \in \mathcal{P}_m$},
\end{equation}
where $d$ is a positive integer with $d \ge 2$.
\end{problem}

In algebraic combinatorics, especially in design theory on the unit sphere $\mathbb{S}^d$, it is customary to use the term \emph{interval design} instead of design.
Interval designs for Gegenbauer measure $(1-t^2)^{(d-3)/2}dt / \int_{-1}^1 (1-t^2)^{(d-3)/2}dt$ are highly motivated by the construction of spherical designs; for example, see Bajnok~\cite{B1992}, Rabau and Bajnok~\cite{RB1991}, Wagner~\cite{W1991}.
As commented by Xiang~\cite{X2022}, it is still an open question to present an explicit construction of rational designs for Gegenbauer measure in general.

In this paper we restrict our attention to density functions $w(t)$ that correspond to \emph{symmetric classical orthogonal polynomials}, namely $w(t) = e^{-t^2}/\sqrt{\pi}$ on $(-\infty,\infty)$ or
$w(t) = (1-t^2)^{\lambda-1/2} / \int_{-1}^1 (1-t^2)^{\lambda-1/2}dt$ on $(-1,1)$ where $\lambda > -1/2$,
and then extensively discuss Problem~\ref{prob:HK}, especially for $m \le 5$.
A motivation for the degree five and four cases comes from the study of
Waring's problem; for example, see Remark~\ref{rem:KawadaWooley2} and Proposition~\ref{prop:KawadaWooley}.

\begin{remark}\label{rem:design}
As seen in Proposition~\ref{prop:HK_1}, an $m$-design with $n$ points is equivalent to Hilbert-Kamke equations (\ref{eq:HK}).
In the framework of the Hilbert-Kamke problem and Waring's problem, it is often allowed that $x_i = x_j$ for distinct $i$ and $j$, whereas this is not often the case in geometric design theory.
On the other hand, it is often the case to only deal with positive rationals $x_i$ for
the Hilbert-Kamke equation, whereas negative rationals are also allowed in design theory; for example see Nathanson~\cite[\S~4.3]{N1996}.
\end{remark}

\subsection{Hilbert-Kamke equations and Hausdorff's construction}\label{subsect:Hausdorff}

For positive integers $n$ and $r$, let $s(n,r)$ be the minimum number $s$ such that $n = n_1^r+\cdots+n_s^r$ for some $s$-tuple $(n_1,\ldots,n_s) \in \nn^s$, and let
\[
g(r) = \max_{n \in \nn} s(n,r).
\]
The classical Waring problem asks whether $g(r)$ is finite and, if $g(r) < \infty$,
what $g(r)$ is. The Lagrange four-square theorem states that $g(2) = 4$.
  
\begin{remark}\label{rem:Gauss}
The Gauss three-square theorem is a refinement of the four-square theorem, stating that a positive integer $n$ is a sum of three squares of integers if and only if $n$ is not of the form $4^e(8k+7)$, where $e$ and $k$ are non-negative integers.
\end{remark}

After Lagrange, in 1859, Liouville made the first progress for $r \ge 3$, who proved that every positive integer can be represented as a sum of at most $53$ fourth powers of integers (biquadrates).
A key for the proof is the polynomial identity
\begin{equation}\label{eq:Liouville}
6(X_1^2 + X_2^2 + X_3^2 + X_4^2)^2 = \sum_{1\le i<j \le 4} \{(X_i+X_j)^4 + (X_i-X_j)^4\}.
\end{equation}

\begin{definition}[cf.~\cite{NS14,R1992}]\label{def:Hilbertidentity}
A polynomial identity of type
\begin{equation}\label{eq:rationalHilbert}
(X_1^2 + \cdots + X_n^2)^r = \sum_{i = 1}^M c_i (a_{i1} X_1 + \cdots + a_{in} X_n)^{2r},
\end{equation}
where $c_i \in \rr_{>0}$ and $a_{ij} \in \rr$, is called a \emph{Hilbert identity}.
In particular, when $c_i \in \qq_{>0}$ and $a_{ij} \in \qq$, this is called a \emph{rational (Hilbert) identity}.
\end{definition}

After Liouville, mathematicians in the late 19th century found similar identities to show the finiteness of $g(r)$ for small $r$.
It was Hilbert who finally and completely established that $g(r) < \infty$ for all $r$.
A key step of Hilbert's proof is Theorem~\ref{thm:Hilbert} below, which was originally stated in $5$ indeterminates $X_1,\ldots,X_5$.

\begin{theorem}[Hilbert's lemma]\label{thm:Hilbert}
For any $n,r \in \nn$, there exist $a_{ij} \in \qq$ and $c_i \in \qq_{>0}$ such that
\begin{equation}\label{eq:Hilbertidentity}
  (X_1^2 + \cdots + X_n^2)^r
     = \sum_{i = 1}^{\binom{2r+n-1}{2r}} c_i (a_{i1} X_1 + \cdots + a_{in} X_n)^{2r}.
\end{equation}
\end{theorem}

Hilbert's proof is known to be fairly long and complicated. For a good expository work concerning Theorem~\ref{thm:Hilbert}, we refer the reader to Pollack~\cite{P09}.

After the pioneering work by Hilbert, the attention of those who had studied Waring's problem turned to the determination of $g(r)$. $g(3)=9$ was established by Wieferich and Kempner from 1909 to 1912.
Their proof involves the formal derivative of Liouville's identity (\ref{eq:Liouville}).
For an early history on Waring's problem, we refer the reader to Dickson's book~\cite[pp.717--725]{D1920}.

Since the 1920s, a powerful technique called the \emph{circle method} was developed by Hardy and Littlewood, which has drastically changed the situation of
Waring's problem.
Surprisingly, over 20 years after the development of the circle method, $g(r)$ was completely determined for all $r \ge 7$ by Dickson, Niven and so on.
Afterwards, $g(6) = 73$ was established by Pillai in 1940, and $g(5) = 37$ by Chen in 1964.

It was Balasubramanian, Deshouillers and Dress who established $g(4) = 19$ in a series of papers published from 1988 to 1993.
Their proof mainly consists of two parts.
The first part makes use of the circle method, in order to show that every positive integer $n > 10^{367}$ is a sum of at most $19$ biquadrates.
The second part is to check on a  large-scale computer system whether each $n \le 10^{367}$ can be actually represented as a sum of at most $19$ biquadrates. 

Now, $10^{367}$ is an astronomical number. It was Kawada and Wooley~\cite{KW1999} who combined the circle method with the polynomial identity
\begin{equation}\label{eq:KawadaWooley}
2(X^2 + XY + Y^2)^2 = X^4 + Y^4 + (X+Y)^4
\end{equation}
and thereby improved $10^{367}$ to $10^{146}$. 

\begin{remark}\label{rem:KawadaWooley2}
As implied by Kawada and Wooley~\cite{KW1999}, the collaboration of the circle method and polynomial identity approach is key in the study of Waring's problem for the degree four case, which exactly applies to Hilbert-Kamke problem,
as will be seen in both Sections~\ref{sect:five-six} and~\ref{sect:conclusion}.
While we mainly consider geometric designs and Hilbert-Kamke equations of degree five in this paper, we point out that, for a positive integer $d$, a set $X$ is a spherical $4$-design on $\mathbb{S}^d$ if and only if $X \cup (-X)$ is a spherical $5$-design on $\mathbb{S}^d$.
\end{remark}

Then what is the polynomial identity (\ref{eq:KawadaWooley})?

\begin{proposition}\label{prop:KawadaWooley}
The polynomial identity (\ref{eq:KawadaWooley}) is transformed to the following Hilbert identity by a regular transformation:
\begin{equation}\label{eq:KawadaWooley1}
(X_1^2 + X_2^2)^2 =
\frac{2}{3} \sum_{k=0}^2 \Big( \cos\Big(\frac{4k + 3}{6}\pi\Big) X_1 + \sin\Big(\frac{4k + 3}{6}\pi\Big) X_2 \Big)^4.
\end{equation}
\end{proposition}
\begin{proof}[Proof of Proposition~\ref{prop:KawadaWooley}]
The result follows by taking the regular transformation 
\[
  X = \frac{-\sqrt{3}X_1+X_2}{\sqrt{6}}, \quad Y = \frac{\sqrt{3}X_1+X_2}{\sqrt{6}}.
\qedhere
\]
\end{proof}

We now go back to Hilbert's lemma (Theorem~\ref{thm:Hilbert}).
After Hilbert, a number of publications have been devoted to the simplification of the original proof of Theorem~\ref{thm:Hilbert}.
A significant contribution was made by Hausdorff~\cite{Hau09}, who first established
\[
\begin{gathered}
{\int \cdots \int}_{\rr^n} (u_1 X_1 + \cdots + u_n X_n)^{2r} \ e^{-(u_1^2+\cdots+u_n^2)} du_1 \cdots du_n \\
= c_{n,r} (X_1^2 + \cdots + X_n^2)^r  \ \text{ for some constant $c_{n,r}$}
\end{gathered}
\]
and then constructed an iterated sum in order to obtain a rational Hilbert identity.
In doing so, Hausdorff showed the following key lemma by making use of the zeros of
the Hermite polynomial of degree $2r+1$.

\begin{theorem}[Hausdorff's lemma]\label{thm:haus}
There exist $x_1, \ldots, x_{2r+1} \in \qq$ and $\lambda_1, \ldots, \lambda_{2r+1} \in \qq_{>0}$ such that
\begin{equation} \label{eq:hausdorff-equation1}
\sum_{i=1}^{2r+1} \lambda_i x_i^j
= \frac{1}{\sqrt{\pi}} \int_{-\infty}^\infty t^j \ e^{-t^2}dt, \qquad j=0,1,\ldots,2r.
\end{equation}
\end{theorem}

As noted after Proposition~\ref{prop:HK_1}, generalizations of Theorem~\ref{thm:haus} are given by Sawa and Uchida~\cite[Theorem~5.5]{SU2020} and Matsumura~\cite{M23}.

A brief explanation of Hausdorff's method is available in Nesterenko~\cite[p.4700]{N06}, where a further refinement of Hausdorff's arguments is discussed. 

The equations (\ref{eq:hausdorff-equation1}) are just weighted Hilbert-Kamke equations (\ref{eq:HK}) for Hermite measure $e^{-t^2}dt/\sqrt{\pi}$. 
It should be noted that, if one could explicitly construct a rational solution of the equations, then Hausdorff's method generates an explicit construction of rational identities, which can also be applied in the realization of solutions of Waring problem.

\subsection{Rational points and the genus of a curve} \label{subsect:rational}

We first review some terminology on algebraic and arithmetic geometry used in Subsection~\ref{subsect:geometric} (see~\cite{HS} for details).

Let $C$ be a projective, irreducible, and smooth curve.
Then the dimension of the vector space of holomorphic differential forms on $C$ is called the \emph{(geometric) genus} of $C$ and denoted by $g(C)$.
When an irreducible curve is not necessarily smooth or projective, the genus of the curve is defined to be the genus of a smooth projective curve which is birational to the curve.

Let $C$ be a curve defined over $\qq$ with $g(C)\ge 2$.
We can explicitly determine $C(\qq)$ with Chabauty's method if $C$ satisfies a certain condition explained below.
The \emph{Jacobian variety} $J$ of $C$ is an Abelian variety of dimension $g(C)$ which is isomorphic to the divisor class group of degree $0$ of $C$ as groups.
By the Mordell-Weil theorem, the group $J(\qq)$ of rational points of $J$ is finitely generated.
The rank of $J(\qq)$ as a $\zz$-module is called the \emph{Mordell-Weil rank} of $J$.
When the Mordell-Weil rank of $J$ is less than the genus of $C$, the rational points on $C$ can be determined by Chabauty's method.
Chabauty's method for genus $2$ curves is implemented as the function \texttt{Chabauty} in the computer algebra system Magma~\cite{Magma}.
We will use this function \texttt{Chabauty} in the proof of Theorem~\ref{thm:Gegenbauer} in Subsection~\ref{subsect:geometric}.

Next, we review some algebro-geometric terminology used in Section~\ref{sect:parametrization}, most of which can be found in \cite{CLO,Hartshorne,Shafarevich}.
Until the end of this subsection we do not assume that varieties are irreducible; we assume that varieties are defined over $\overline{\qq}$ unless otherwise stated.

For homogeneous polynomials $F_1,\dotsc,F_r\in\overline{\qq}[X_0,\dotsc,X_n]$, we define
\[
    Z(F_1, \dotsc, F_r) = \{\bm{x} \in \mathbb{P}^n
    \mid F_1(\bm{x}) = \dotsb = F_r(\bm{x}) = 0\}.
\]
A subset $X\subset\mathbb{P}^n$ is called
a \emph{projective variety} if there exist homogeneous polynomials
$F_1,\dotsc,F_r\in\overline{\qq}[X_0,\dotsc,X_n]$
such that $X=Z(F_1,\dotsc,F_r)$.
For a projective variety $X\subset\mathbb{P}^n$, the \emph{ideal} (or \emph{annihilating ideal}) of $X$ is defined by
\[
    I(X) = \{F \in \overline{\qq}[X_0,\dotsc,X_n]
    \mid F(\bm{x}) = 0 \text{ for all } \bm{x} \in X\}.
\]

Let $F_1,\dotsc,F_r\in\overline{\qq}[X_0,\dotsc,X_n]$ be homogeneous polynomials.
We define the \emph{Jacobian matrix} $J(\bm{x})$ by
\[
    J(\bm{x}) = \left(\dfrac{\partial F_i}{X_j}(\bm{x})
    \right)_{\substack{1\le i\le r \\ 0\le j\le n}}.
\]
Following \cite{BH}, we call the set $\{F_1,\dotsc,F_r\}$
a \emph{nonsingular system} if $\rank J(\bm{x})=r$
for every nonzero $\bm{x}\in\overline{\qq}^n$ 
satisfying $F_i(\bm{x})=0$, $i=1,\dotsc,r$.

A projective variety $X\subset\mathbb{P}^n$ of dimension $n-r$
is called a \emph{complete intersection} if $I(X)$ can be
generated by $r$ homogeneous polynomials.

\begin{lemma}\label{lemma:BH}
Let $\{F_1,\dotsc,F_r\}$ be a nonsingular system of homogeneous polynomials and let $X=Z(F_1,\dotsc,F_r)\subset\mathbb{P}^n$.
Then $I(X)$ is generated by $F_1,\dotsc,F_r$,
and $X$ is a smooth complete intersection of dimension $n-r$.
Moreover, if $n-r\ge 1$, then $X$ is irreducible.
\end{lemma}
\begin{proof}[Proof of Lemma~\ref{lemma:BH}]
See {\cite[Lemma~3.2]{BH}}. Note that it is proved that $X$ is
geometrically integral in \cite[Lemma~3.2]{BH}, which implies that $X$ is irreducible over $\overline{\qq}$.
\end{proof}

Let $C$ be an irreducible curve.
The arithmetic genus of $C$ is denoted by $p_a(C)$.
If $C$ is smooth, then $p_a(C)$ coincides with the geometric genus of $C$.
Let $\tilde{C}$ be the normalization of $C$.
Then $\tilde{C}$ is a smooth curve.
We call $p_a(\tilde{C})$ the genus of $C$.

\begin{lemma}\label{lemma:genus_CI}
Let $F_1,F_2\in\overline{\qq}[X_0,X_1,X_2,X_3]$ be homogeneous polynomials with $\deg F_1=d$ and $\deg F_2=e$.
Let $C=Z(F_1,F_2)\subset\mathbb{P}^3$.
If $C$ is a complete intersection, then we have $p_a(C)=de(d+e-4)/2+1$.
\end{lemma}
\begin{proof}[Proof of Lemma~\ref{lemma:genus_CI}]
See \cite[Exercise~I.7.2~(d)]{Hartshorne}.
\end{proof}
\begin{remark}
We can calculate the arithmetic genus of a complete intersection of any dimension.
For example, see \cite{AS}.
\end{remark}

\begin{lemma}\label{lemma:arithmetic_genus}
Let $C$ be an irreducible curve and let $\tilde{C}$ be the
normalization of $C$.
For $P\in C$, we write $\mathcal{O}_P$ and $\tilde{\mathcal{O}}_P$
for the local ring at $P$ and the integral closure of
$\mathcal{O}_P$, respectively.
Let $\delta_P=\operatorname{length}
(\tilde{\mathcal{O}}_P/\mathcal{O}_P)$.
Then we have
\[
    p_a(C) = p_a(\tilde{C}) + \sum_{P\in C} \delta_P.
\]
\end{lemma}
\begin{proof}[Proof of Lemma~\ref{lemma:arithmetic_genus}]
See \cite[Exercise~IV.1.8~(a)]{Hartshorne} or \cite[Proposition~7.5.4]{Liu}.
\end{proof}

\subsection{A geometric characterization of Hilbert-Kamke equations} \label{subsect:geometric}

The study of antipodal rational designs for symmetric measures is naturally motivated by the following fact. 

\begin{theorem}[\cite{MMS2024}]\label{thm:antipodality}
Let $m=2\ell-1$ be an odd integer and $n$ be a positive integer with $n \le m+1$.
Then the Hilbert-Kamke equations~\eqref{eq:HK} for symmetric classical measures have only antipodal solutions.
\end{theorem}

This result was first observed by the present authors only for $(m,n) = (5,6)$.
Afterwards, the general version was established by Misawa et al.~\cite{MMS2024}.
The proof of Theorem~\ref{thm:antipodality} makes use of Newton's identity on elementary symmetric polynomials and power sums.

Now, we have a look at applications of Theorem~\ref{thm:antipodality} for $(m,n) = (5,3),(5,4),(5,5)$.
First we consider a $5$-design with $n$ rational points for Gegenbauer measure $(1-t^2)^{\lambda-1/2}dt/\int_{-1}^1 (1-t^2)^{\lambda-1/2} dt$, namely
\begin{equation}\label{eq:QF3-5-Chebyshev}
    \frac{1}{n} \sum_{i=1}^n x_i^k = \frac{1}{\int_{-1}^1 (1-t^2)^{\lambda-1/2} dt} \int_{-1}^1 t^k (1-t^2)^{\lambda-1/2} dt, \quad
    k = 1,2,3,4,5.
\end{equation}
The moments $a_k$ are computed as follows.
For a nonnegative integer $k$, we have
\[
    \int_{-1}^1 t^{2k} (1-t^2)^{\lambda-1/2} dt
    = 2 \int_0^1 t^{2k} (1-t^2)^{\lambda-1/2} dt.
\]
Putting $x=t^2$, we have
\begin{equation}
\label{eq:moment_Gegenbauer}
\begin{gathered}
   2 \int_0^1 t^{2k} (1-t^2)^{\lambda-1/2} dt
    = \int_0^1 x^{k-1/2} (1-x)^{\lambda-1/2} dx \\
    = B\left(k+\frac{1}{2},\lambda+\frac{1}{2}\right)
    = \frac{\Gamma(k+1/2)\Gamma(\lambda+1/2)}{\Gamma(k+\lambda+1)},
\end{gathered}
\end{equation}
where $B(t)$ and $\Gamma(t)$ are the beta and gamma functions,
respectively.
Then
\begin{align}
    a_2 &= \frac{\Gamma(3/2)\Gamma(\lambda+1/2)}{\Gamma(\lambda+2)}
    \cdot \frac{\Gamma(\lambda+1)}{\Gamma(1/2)\Gamma(\lambda+1/2)}
    = \frac{1}{2(\lambda+1)}, \label{eq:moment_Gegenbauer2} \\
    a_4 &= \frac{\Gamma(5/2)\Gamma(\lambda+1/2)}{\Gamma(\lambda+3)}
    \cdot \frac{\Gamma(\lambda+1)}{\Gamma(1/2)\Gamma(\lambda+1/2)}
    = \frac{3}{4(\lambda+1)(\lambda+2)}. \label{eq:moment_Gegenbauer3}
\end{align}
We also have $a_{2\ell-1} = 0$ for $\ell=1,2,\ldots$.

By Theorem~\ref{thm:antipodality}, $x_1,\dotsc,x_n$ are antipodal and hence we may assume that $x_i=-x_{n+1-i}$ for $i=1,\dotsc,n$.
For $n=3$, the solutions of \eqref{eq:QF3-5-Chebyshev} are $\lambda=0$ and $x_1=\pm\sqrt{3}/2$.
Hence there do not exist $5$-designs with $3$ rational points.
For $n=4,5$, after eliminating $\lambda$, the equations \eqref{eq:QF3-5-Chebyshev} can be reduced to finding rational points of the curves
\begin{gather}
    2 (x_1^4 + x_2^4) (x_1^2 + x_2^2 + 1) = 3 (x_1^2 + x_2^2)^2,
    \label{eq:n4} \\
    (x_1^4 + x_2^4) (4 x_1^2 + 4 x_2^2 + 5) = 6 (x_1^2 + x_2^2)^2,
    \label{eq:n5}
\end{gather}
respectively.

\begin{theorem}
\hangindent\leftmargini
\label{thm:Gegenbauer}
\textup{(i)} The rational points of \eqref{eq:n4} are completely classified by
\[
    (x_1,x_2) = (-1,-1), (-1,1), (0,0), (1,-1), (1,1).
\]
In particular there do not exist rational $5$-designs with $4$ points for Gegenbauer measure $(1-t^2)^{\lambda-1/2}dt/\int_{-1}^1 (1-t^2)^{\lambda-1/2} dt$ on $(-1,1)$.
\begin{enumerate}
\setcounter{enumi}{1}
\item The rational points of \eqref{eq:n5} are completely classified by
\[
    (x_1,x_2) = \left(-\frac{1}{2},0\right), \left(0,-\frac{1}{2}\right), (0,0), \left(0,\frac{1}{2}\right), \left(\frac{1}{2},0\right).
\]
In particular there do not exist rational $5$-designs with $5$ points for Gegenbauer measure $(1-t^2)^{\lambda-1/2}dt/\int_{-1}^1 (1-t^2)^{\lambda-1/2} dt$ on $(-1,1)$.
\end{enumerate}
\end{theorem}

To prove Theorem~\ref{thm:Gegenbauer}, we will use the function \texttt{Chabauty} in Magma~\cite{Magma}, which implements Chabauty's method
(see Subsection~\ref{subsect:rational}).

\begin{proof}[Proof of Theorem~\ref{thm:Gegenbauer}]
We first consider the equation \eqref{eq:n4}.
Assume that $x_1\neq 0$ and set $t=x_2^2/x_1^2$.
Eliminating $x_2$ from \eqref{eq:n4}, we have
\[
    2 (x_1^4 + t^2 x_1^4) (x_1^2 + t x_1^2 + 1)
    = 3 (x_1^2 + t x_1^2)^2.
\]
Dividing both sides by $x_1^4$, we have
\[
    2 (1 + t^2) ((1 + t) x_1^2 + 1) = 3 (1 + t)^2,
\]
that is,
\[
    2 (1 + t^2) (1 + t) x_1^2 = t^2 + 6 t + 1.
\]
Putting $x=2t$ and $y = 8 (1 + t^2) (1 + t) x_1$, we obtain a genus $2$ curve
\[
    C_1 \colon y^2 = x^5 + 14 x^4 + 32 x^3 + 64 x^2 + 112 x + 32.
\]
We can verify that the Mordell-Weil rank of the Jacobian variety of $C_1$ is equal to $1$ with Magma.
Then, by using Magma's function \texttt{Chabauty}, we have
\[
    C_1(\qq) = \{ \infty, (-2,0), (2,-32), (2,32) \},
\]
where $\infty$ is the point at infinity.
Since $t=x_2^2/x_1^2\ge 0$, we have $x\ge 0$.
Hence it is sufficient to consider the points
$(2,-32)$ and $(2,32)$ on $C_1$.
When $(x,y)=(2,-32)$, we obtain $t=1$ and $x_1=-1$.
Thus we have $x_2^2=1$.
Similarly, we obtain $x_1=1$ and $x_2^2=1$ when $(x,y)=(2,32)$.
Therefore we have $(x_1,x_2)=(-1,-1),(-1,1),(1,-1),(1,1)$.
When $x_1=0$, substituting it into \eqref{eq:n4}, we have
\[
    2 x_2^4 (x_2^2 + 1) = 3 x_2^4.
\]
Then we have $x_2=0,\pm 1/\sqrt{2}$.
Since $x_2$ is rational, we have $(x_1,x_2)=(0,0)$.

Next we consider the equation \eqref{eq:n5}.
We assume that $x_1\neq 0$ and set $t=x_2^2/x_1^2$.
Eliminating $x_2$ from \eqref{eq:n5}, we have
\[
    (x_1^4 + t^2 x_1^4) (4 x_1^2 + 4 t x_1^2 + 5)
    = 6 (x_1^2 + t x_1^2)^2.
\]
Dividing both sides by $x_1^4$, we have
\[
    (1 + t^2) (4 (1 + t) x_1^2 + 5) = 6 (1 + t)^2,
\]
that is,
\[
    4 (1 + t^2) (1 + t) x_1^2 = t^2 + 12 t + 1.
\]
Putting $x=t$ and $y = 2 (1 + t^2) (1 + t) x_1$, we obtain a genus $2$ curve
\[
    C_2 \colon y^2 = x^5 + 13 x^4 + 14 x^3 + 14 x^2 + 13 x + 1.
\]
We can verify that the Mordell-Weil rank of the Jacobian variety of $C_2$ is equal to $1$ with Magma.
Then, by using Magma's function \texttt{Chabauty}, we have
\[
    C_2(\qq) = \{ \infty, (-1,0), (0,-1), (0,1) \}.
\]
Since $t=x_2^2/x_1^2\ge 0$, we have $x\ge 0$.
Hence we have $t=x=0$ and $x_2=0$.
By \eqref{eq:n5}, we have $x_1^4(4x_1^2+5)=6x_1^4$.
Since $x_1\neq 0$, we have $4x_1^2+5=6$.
Therefore we obtain $x_1=\pm 1/2$.
When $x_1=0$, by a similar argument, we have $x_2=0, \pm 1/2$.
\end{proof}

Next we consider the Hermite measure case for $n \in \{3,4,5\}$.
Then the moments $a_k$ are given as
\begin{equation}\label{eq:Hermite_moment}
	a_{2k} = \frac{(2k)!}{2^{2k} k!},\quad a_{2k+1} = 0, \quad k=0,1,\ldots
\end{equation}
By Theorem~\ref{thm:antipodality}, if 
\[
    \frac{1}{n} \sum_{i=1}^n x_i^k = \frac{1}{\sqrt{\pi}} \int_{-\infty}^\infty t^k e^{-t^2}dt, \quad
    k = 1,2,3,4,5,
\]
then $x_i$ are antipodal, say $x_i = -x_{n+1-i}$.
As in the Chebyshev-Gauss quadrature (\ref{eq:GaussChebyshev}), a $5$-design for $n=3$ is uniquely determined by the zeros of Hermite polynomial
\[
H_3(t) = t^3-3t = t(t-\sqrt{3})(t+\sqrt{3}),
\]
which is not rational.
By (\ref{eq:Hermite_moment}), a $5$-design for $n = 4$ or $n = 5$ is reduced to the equations
\[
x_1^2 + x_2^2 = 1, \quad x_1^4 + x_2^4 = \frac{3}{2}
\]
or
\[
x_1^2 + x_2^2 = \frac{5}{4}, \quad x_1^4 + x_2^4 = \frac{15}{8}.
\]
In any case, the solutions are irrational. In summary, we obtain the following result.

\begin{theorem}\label{thm:Hermite}
For $n = 3,4,5$, there do not exist rational $5$-designs with $n$ points for $e^{-t^2}dt/\sqrt{\pi}$ on $(-\infty,\infty)$.
\end{theorem}

\begin{remark}\label{rem:Krylov}
For $n=4,5$, there do not exist $5$-designs with $n$ (not necessarily rational) points for $e^{-t^2}dt/\sqrt{\pi}$ on $(-\infty,\infty)$; more details can be found in Gautschi~\cite{G1975}.
\end{remark}

\section{Proof of Theorem~\ref{thm:deg3}} \label{sect:small_deg}

Let us consider the Hermite measure $e^{-t^2}dt/\sqrt{\pi}$ on $(-\infty,\infty)$.
By (\ref{eq:Hermite_moment}), an antipodal $3$-design of type
\begin{align}
\frac{1}{n} \sum_{i=1}^{n} f(x_i)
&= \frac{1}{\sqrt{\pi}} \int_{-\infty}^{\infty} f(t) e^{-t^2}dt
\label{eq:Hermite_even}
\end{align}
is equivalent to a disjoint solution of the Diophantine equations
\begin{align}
   x_1 +   x_2  + \cdots  +   x_n  & = 0,              \label{eq:Hermite_deg1} \\
x_1^2 + x_2^2 + \cdots  + x_n^2 & = \frac{n}{2}. \label{eq:Hermite_deg2}
\end{align}

It is obvious that the equations (\ref{eq:Hermite_deg1}) and (\ref{eq:Hermite_deg2}) have no rational solutions for $n=1,2$. Thus, the smallest case to be solved is $n=3$.

\begin{lemma}\label{lem:deg3_Hermite}
There do not exist $3$-designs with $3$ rational points for $e^{-t^2}dt/\sqrt{\pi}$ on $(-\infty,\infty)$.
\end{lemma}
\begin{proof}[Proof of Lemma~\ref{lem:deg3_Hermite}]
Suppose contrary.
Then by Theorem~\ref{thm:antipodality}, 
\[
x_1^k + x_2^k + x_3^k = 3a_k, \quad k=1,2,3
\]
have only antipodal solutions.
Without loss of generality, we may let $x_1=0$ and $x_2 = -x_3$.
This is however impossible, since $3/2 = 3a_2 = 2x_2^2$ by (\ref{eq:Hermite_moment}).
\end{proof}

\begin{lemma}\label{lem:deg4_Hermite}
For any positive integer $n \ge 4$ except possibly for $n = 7$, there exists an antipodal $3$-design with $n$ rational points for $e^{-t^2}dt/\sqrt{\pi}$ on $(-\infty,\infty)$.
\end{lemma}
\begin{proof}[Proof of Lemma~\ref{lem:deg4_Hermite}]
For $n=4$, the equations (\ref{eq:Hermite_deg1}) and (\ref{eq:Hermite_deg2}) have an antipodal solution, say $x_3 = -x_1$ and $x_4 = -x_2$, which can be reduced to finding rational points on the unit circle
$\mathbb{S}^1 = \{(x,y) \in \mathbb{R}^2 \mid x^2+y^2=1 \}$.
All rational points of $\mathbb{S}^1$ are completely classified by
\begin{equation}\label{eq:four}
x_1 = \frac{1-t^2}{1+t^2}, \quad x_2 = \frac{2t}{1+t^2}, \quad t \in \qq.
\end{equation}

Similarly for $n=5$, the equations (\ref{eq:Hermite_deg1}) and (\ref{eq:Hermite_deg2}) have an antipodal solution, say $x_4=-x_1, x_5=-x_2, x_3=0$, which can be  reduced to finding rational points on the circle $x_1^2+x_2^2=\frac{5}{4}$.
All rational points on this circle are completely classified by
\begin{equation}\label{eq:five}
x_1 = \frac{t^2-t-1}{1+t^2}, \quad x_2 = \frac{-t^2-4t+1}{2+2t^2}, \quad t \in \qq.
\end{equation}

As already seen in Example~\ref{ex:Chebyshev1},
\[
(x_1,x_2,x_3,x_4,x_5,x_6) = \Big(\frac{1}{7}, \frac{11}{14}, \frac{-13}{14}, -\frac{1}
{7}, -\frac{11}{14}, -\frac{13}{14} \Big)
\]
presents an antipodal $3$-design ($5$-design) with $6$ rational points.
Note that infinitely many disjoint rational points $(x_1,x_2)$ on circles can be chosen from (\ref{eq:four}) and (\ref{eq:five}).
Therefore, by Proposition~\ref{prop:composition}, we obtain an antipodal rational $3$-design
for $n \ge 4$ except possibly for $n=7$.
\end{proof}

The following shows the nonexistence of antipodal $3$-designs for $n = 7$.

\begin{lemma}\label{lem:deg7_Hermite}
A positive rational number $b$ is a sum of three squares of rational numbers if and only if $b$ is not of the form $r^2(8k+7)$, where $r$ is a positive rational number and $k$ is a nonnegative integer.
\end{lemma}
\begin{proof}[Proof of Lemma~\ref{lem:deg7_Hermite}]
Suppose that $b$ is not of the form $r^2(8k+7)$.
We can write $b = c/d^2$, where $c$ and $d$ are positive integers.
Then there exist unique
non-negative integers $e,k$ and $l$, for which $c = 4^e(8k+l)$ and $l \in \{1,2,3,5,6,7\}$.
Since $b$ is not of the form $r^2(8k+7)$, we have $l \ne 7$.
By Remark~\ref{rem:Gauss}, there exist integers $l_1,l_2,l_3$ such that $l_1^2+l_2^2+l_3^2=c$.
Hence we have $(l_1/d)^2 + (l_2/d)^2 + (l_3/d)^2 = b$.

Conversely, suppose that $b = r^2(8k+7)$ and there exist rational numbers $m_1,m_2,m_3$ for which $m_1^2 + m_2^2 + m_3^2 = b$.
We can write $m_i = n_i/d$ and $r = c/d$, where $n_1,n_2,n_3,c,d$ are integers.
Then we have $n_1^2 + n_2^2 + n_3^2 = c^2(8k+7)$.
Let $c = 2^e c'$, where $e$ is a non-negative integer and $c'$ is an odd integer.
Since $(c')^2 \equiv 1 \pmod{8}$, we have $c^2 (8k+7) = 4^e(8k'+7)$, where $k'$ is a non-negative integer.
This is a contradiction to the comment in Remark~\ref{rem:Gauss}.
\end{proof}

We are now in a position to complete the proof of Theorem~\ref{thm:deg3}.

\noindent
\begin{proof}[Proof of Theorem~\ref{thm:deg3}]
In the Hermite measure case, combining Lemmas~\ref{lem:deg3_Hermite},~\ref{lem:deg4_Hermite} and~\ref{lem:deg7_Hermite}, we obtain the desired result.
\end{proof}

We close this section by discussing an analogue of Theorem~\ref{thm:deg3} for the Gegenbauer measure $(1-t^2)^{\lambda-1/2}dt/\int_{-1}^1 (1-t^2)^{\lambda-1/2} dt$ with $\lambda > -1/2$.
For example when $\lambda=0$, the corresponding measure is the Chebyshev measure $(1-t^2)^{-1/2}dt/\pi$.
By (\ref{eq:moment_Gegenbauer2}), the second moment for the Chebyshev measure equals $a_2 = 1/2$, which coincides with that for Hermite measure.
Hence, as a corollary of Theorem~\ref{thm:deg3}, we obtain the following result.

\begin{corollary}\label{cor:deg3}
There exists an antipodal $3$-design with $n$ rational points for the Chebyshev measure $(1-t^2)^{-1/2}dt/\pi$ on $(-1,1)$ if and only if $n \notin \{1,2,3,7\}$.
\end{corollary}

\begin{remark}\label{rem:twothree}
An analogue of Theorem~\ref{thm:deg3} can be obtained for many values of $\lambda > -1/2$, although there is not always a systematic way of treating the degree three case.
\end{remark}

\section{Proof of Theorem~\ref{thm:parametrization0}} \label{sect:parametrization}

By Theorem~\ref{thm:antipodality}, a $5$-design with $6$ points for symmetric classical measure is always antipodal, which can be reduced to the equations
\begin{equation}\label{eq:HK4}
   \begin{split}
       x_1^2 + x_2^2 + x_3^2 &= 3 a_2, \\
       x_1^4 + x_2^4 + x_3^4 &= 3 a_4.
   \end{split}
\end{equation}
By substituting $x_i=X_i/X_0$ in \eqref{eq:HK4}, we have the system of homogeneous equations
\begin{align*}
   X_1^2 + X_2^2 + X_3^2 - 3 a_2 X_0^2 &= 0, \\
   X_1^4 + X_2^4 + X_3^4 - 3 a_4 X_0^4 &= 0.
\end{align*}
We consider curves defined by such equations.

In the following, we use the algebro-geometric terminology introduced in Subsection~\ref{subsect:rational}.
We write $\bm{X}=(X_0,X_1,X_2,X_3)$ and define
\begin{equation}\label{eq:Jacob0} \begin{gathered}
    F_1(\bm{X}) = X_1^2 + X_2^2 + X_3^2 - \alpha X_0^2, \\
    F_2(\bm{X}) = X_1^4 + X_2^4 + X_3^4 - \beta X_0^4,
\end{gathered} \end{equation}
where $\alpha,\beta\in\overline{\qq}^\times$.
Let $C=Z(F_1,F_2)\subset\mathbb{P}^3$.

\begin{lemma}\label{lemma:nonsingular}
The set $\{F_1,F_2\}$ is a nonsingular system if and only if $\alpha^2\notin\{\beta,2\beta,3\beta\}$.
\end{lemma}
\begin{proof}[Proof of Lemma~\ref{lemma:nonsingular}]
The Jacobian matrix of
the system $F_1(\bm{X}) = F_2(\bm{X}) = 0$ is calculated as follows:
\[
    J(X_0,X_1,X_2,X_3) = \begin{pmatrix}
        -2 \alpha X_0 & 2 X_1 & 2 X_2 & 2 X_3 \\
        -4 \beta X_0^3 & 4 X_1^3 & 4 X_2^3 & 4 X_3^3
    \end{pmatrix}.
\]
Assume that $\{F_1,F_2\}$ is not a nonsingular system.
Then $\rank J(\bm{x})\neq 2$ for some nonzero $\bm{x}=(x_0:x_1:x_2:x_3)\in\overline{\qq}^4$ satisfying
\begin{align}
    x_1^2 + x_2^2 + x_3^2 - \alpha x_0^2 &= 0, \label{eq:quadratic} \\
    x_1^4 + x_2^4 + x_3^4 - \beta x_0^4 &= 0. \label{eq:quartic}
\end{align}
We prove that $\alpha^2\in\{\beta,2\beta,3\beta\}$.
Since all $2\times 2$ minors of $J(\bm{x})$ vanish,
we have
\begin{align}
    \alpha x_0 x_i^3 &= \beta x_0^3 x_i, \quad i = 1, 2, 3, \label{eq:relation_0i} \\
    x_i x_j^3 &= x_i^3 x_j, \quad 1 \le i < j \le 3. \label{eq:relation_ij}
\end{align}

\begin{enumerate}[label=(\Roman*)]
\item We first assume that $x_0\neq 0$.
For any $1\le i<j\le 3$, the relation \eqref{eq:relation_ij} implies
at least one of the relations $x_i=0$, $x_j=0$, and $x_i^2=x_j^2$.
\begin{enumerate}[label=(\alph*)]
\item When $x_1^2=x_2^2=x_3^2$, since $\bm{x}$ is nonzero,
$x_1, x_2, x_3$ are all nonzero.
By \eqref{eq:quadratic} and \eqref{eq:quartic}, we have
\[
    3 x_1^2 - \alpha x_0^2 = 0, \quad
    3 x_1^4 - \beta x_0^4 = 0.
\]
Eliminating $x_0$ and $x_1$, we obtain $\alpha^2=3\beta$.
\item When $x_i^2\neq x_j^2$ for some $i$ and $j$,
we may assume that $x_2^2\neq x_3^2$ without loss of generality.
Then we have $x_2 x_3=0$ by \eqref{eq:relation_ij}.
We may assume that $x_3=0$.
\begin{enumerate}[label=(\roman*)]
\item When $x_1^2=x_2^2$, by \eqref{eq:quadratic} and \eqref{eq:quartic}, we have
\[
    2 x_1^2 - \alpha x_0^2 = 0, \quad
    2 x_1^4 - \beta x_0^4 = 0.
\]
Eliminating $x_0$ and $x_1$, we obtain $\alpha^2=2\beta$.
\item When $x_1^2\neq x_2^2$, we have $x_1 x_2=0$ by \eqref{eq:relation_ij}.
We may assume that $x_2=0$.
Then, by \eqref{eq:quadratic} and \eqref{eq:quartic}, we have
\[
    x_1^2 - \alpha x_0^2 = 0, \quad
    x_1^4 - \beta x_0^4 = 0.
\]
Eliminating $x_0$ and $x_1$, we obtain $\alpha^2=\beta$.
\end{enumerate}
\end{enumerate}
\item Next, we assume that $x_0=0$.
\begin{enumerate}[label=(\alph*)]
\item When $x_1^2=x_2^2=x_3^2$,
we have $3x_1^2=0$ by \eqref{eq:quadratic}.
Hence we have $x_0=x_1=x_2=x_3=0$,
which contradicts the assumption that $\bm{x}$ is nonzero.
\item When $x_i^2\neq x_j^2$ for some $i$ and $j$,
we may assume that $x_2^2\neq x_3^2$.
Then we have $x_2 x_3=0$ by \eqref{eq:relation_ij}.
We may assume that $x_3=0$.
\begin{enumerate}[label=(\roman*)]
\item When $x_1^2=x_2^2$,
we have $2x_1^2=0$ by \eqref{eq:quadratic}.
Hence we have $x_0=x_1=x_2=x_3=0$, which is a contradiction.
\item When $x_1^2\neq x_2^2$, we have $x_1 x_2=0$ by \eqref{eq:relation_ij}.
We may assume that $x_2=0$.
Then we have $x_1^2=0$ by \eqref{eq:quadratic}.
Hence we have $x_0=x_1=x_2=x_3=0$, which is a contradiction.
\end{enumerate}
\end{enumerate}
\end{enumerate}

Conversely, we assume that $\alpha^2\in\{\beta,2\beta,3\beta\}$.
Then, the following points $\bm{x}$ satisfy \eqref{eq:quadratic},
\eqref{eq:quartic}, and $\rank J(\bm{x})\neq 2$,
where the $\pm$ signs are chosen independently.
\begin{enumerate}[label=(\Roman*)]
\item When $\alpha^2=\beta$, $\bm{x}=(1 : \pm\sqrt{\alpha} : 0 : 0), 
(1 : 0 : \pm\sqrt{\alpha} : 0), (1 : 0 : 0 : \pm\sqrt{\alpha})$.
\item When $\alpha^2=2\beta$, $\bm{x}=(\sqrt{2} : \pm\sqrt{\alpha} : \pm\sqrt{\alpha} : 0), 
(\sqrt{2} : \pm\sqrt{\alpha} : 0 : \pm\sqrt{\alpha}), (\sqrt{2} : 0 : \pm\sqrt{\alpha} : \pm\sqrt{\alpha})$.
\item When $\alpha^2=3\beta$,
$\bm{x}=(\sqrt{3} : \pm\sqrt{\alpha} : \pm\sqrt{\alpha} : \pm\sqrt{\alpha})$.
\end{enumerate}
This completes the proof.
\end{proof}

\begin{lemma}\label{lemma:curve_HK}
Let $C$ be a curve defined by the equations $F_1(\bm{X}) = 0$ and $F_2(\bm{X}) = 0$.
\begin{enumerate}
\item\label{item:general} If $\alpha^2\neq \beta, 2\beta, 3\beta$, then $C$ is an irreducible smooth curve of genus $9$. 
\item\label{item:beta} If $\alpha^2=\beta$, then $C$ is an irreducible singular curve with $6$ singular points of genus $3$.
\item\label{item:2beta} If $\alpha^2=2\beta$, then $C$ has $4$ irreducible components and each component is a smooth curve of genus $0$.
\item\label{item:3beta} If $\alpha^2=3\beta$, then $C$ has $2$ irreducible components and each component is a smooth curve of genus $1$.
\end{enumerate}
\end{lemma}
\begin{proof}[Proof of Lemma~\ref{lemma:curve_HK}]
\ref{item:general} By Lemma~\ref{lemma:nonsingular},
$\{F_1,F_2\}$ is a nonsingular system.
Hence, $C$ is a smooth complete intersection of a quadric and a quartic by Lemma~\ref{lemma:BH}.
Therefore, $C$ has genus $9$ by Lemma~\ref{lemma:genus_CI}.

\noindent \ref{item:beta}
We may assume that $\alpha=1$ by replacing $\sqrt{\alpha} X_0$ with $X_0$.
We can verify that the curve $C$ is irreducible over $\overline{\qq}$ with Magma's function \texttt{IsAbsolutelyIrreducible}.
The curve $C$ has $6$ singular points by the proof of Lemma~\ref{lemma:nonsingular}.
By symmetry, $\delta_P$ for all singular points $P$ have the same value $\delta$.
Hence, by Lemma~\ref{lemma:arithmetic_genus},
we have $p_a(C)=p_a(\tilde{C})+6\delta$,
where $\tilde{C}$ is the normalization of $C$.
By Lemma~\ref{lemma:genus_CI}, we have $p_a(C)=9$.
Since $\tilde{C}$ is smooth, we have $p_a(\tilde{C})\ge 0$.
Therefore we have $p_a(\tilde{C})=9-6\delta\ge 0$.
Since $\delta$ is a positive integer,
we have $p_a(\tilde{C})=9-6\cdot 1=3$.
    
\noindent \ref{item:2beta}
The proof is similar to the one for Lemma~\ref{lem:5-6-1}. Let $\bm{x}=(x_0:x_1:x_2:x_3)\in C$. Then we have
$F_1(\bm{x})=F_2(\bm{x})=0$, that is,
\begin{align}
    x_1^2 + x_2^2 + x_3^2 &= \alpha x_0^2, \label{eq:F1} \\
    x_1^4 + x_2^4 + x_3^4 &= \beta x_0^4. \label{eq:F2}
\end{align}
By squaring both sides of \eqref{eq:F1}, we have
\begin{equation}\label{eq:F1_2}
    x_1^4 + x_2^4 + x_3^4
    + 2 x_1^2 x_2^2 + 2 x_1^2 x_3^2 + 2 x_2^2 x_3^2
    = \alpha^2 x_0^4.
\end{equation}
Since $\alpha^2=2\beta$, by \eqref{eq:F2} and \eqref{eq:F1_2},
we obtain
\begin{align*}
    0 &= 2 (x_1^4 + x_2^4 + x_3^4)
    - (x_1^4 + x_2^4 + x_3^4
    + 2 x_1^2 x_2^2 + 2 x_1^2 x_3^2 + 2 x_2^2 x_3^2) \\
    &= x_1^4 + x_2^4 + x_3^4
    - 2 x_1^2 x_2^2 - 2 x_1^2 x_3^2 - 2 x_2^2 x_3^2 \\
    &= (x_1 + x_2 + x_3) (x_1 + x_2 - x_3)
    (x_1 - x_2 + x_3) (x_1 - x_2 - x_3).
\end{align*}
Let $G_{ij}(\bm{X})=X_1+(-1)^i X_2+(-1)^j X_3$ for $i,j=0,1$.
Then we have $G_{ij}(\bm{x})=0$ for some $i$ and $j$.

Conversely, if $F_1(\bm{x})=G_{ij}(\bm{x})=0$ for some $i$ and $j$,
we have $F_1(\bm{x})=F_2(\bm{x})=0$.
Therefore we have $C=Z(F_1,G_{00})\cup Z(F_1,G_{01}) \cup
Z(F_1,G_{10}) \cup Z(F_1,G_{11})$.

By eliminating $X_3$ from $F_1(\bm{X})=G_{ij}(\bm{X})=0$, we have
\[
    2 X_1^2 \pm 2 X_1 X_2 + 2 X_2^2 - \alpha X_0^2 = 0.
\]
This equation defines a smooth conic in $\mathbb{P}^2$,
which is a curve of genus $0$.
Since $G_{ij}$ is linear,
$Z(F_1,G_{ij})$ is a smooth curve of genus $0$.

\noindent \ref{item:3beta}
Although the statement can be proved by using a computer algebra system such as Magma,
we prove it by a direct calculation below.

Let $\bm{x}=(x_0:x_1:x_2:x_3)\in C$.
Then we have \eqref{eq:F1}--\eqref{eq:F1_2}.
Since $\alpha^2=3\beta$, by \eqref{eq:F2} and \eqref{eq:F1_2},
we obtain
\begin{align*}
    0 &= 3 (x_1^4 + x_2^4 + x_3^4)
    - (x_1^4 + x_2^4 + x_3^4
    + 2 x_1^2 x_2^2 + 2 x_1^2 x_3^2 + 2 x_2^2 x_3^2) \\
    &= 2 (x_1^4 + x_2^4 + x_3^4
    - x_1^2 x_2^2 - x_1^2 x_3^2 - x_2^2 x_3^2) \\
    &= 2 (x_1^2 + \omega x_2^2 + \omega^2 x_3^2)
    (x_1^2 + \omega^2 x_2^2 + \omega x_3^2),
\end{align*}
where $\omega$ is a primitive third root of unity.
Hence, setting $G_1(\bm{X})=X_1^2+\omega X_2^2+\omega^2 X_3^2$
and $G_2(\bm{X})=X_1^2+\omega^2 X_2^2+\omega X_3^2$,
we have either $G_1(\bm{x})=0$ or $G_2(\bm{x})=0$.

Conversely, for $i=1,2$,
if $F_1(\bm{x})=G_i(\bm{x})=0$,
then we have $F_1(\bm{x})=F_2(\bm{x})=0$.
Therefore we have $C=Z(F_1,G_1)\cup Z(F_1,G_2)$.

We prove that $\{F_1,G_1\}$ is a nonsingular system.
The Jacobian matrix of the system $\{F_1,G_1\}$ is
\[
    J(X_0,X_1,X_2,X_3) = \begin{pmatrix}
        -2 \alpha X_0 & 2 X_1 & 2 X_2 & 2 X_3 \\
        0 & 2 X_1 & 2 \omega X_2 & 2 \omega^2 X_3
    \end{pmatrix}.
\]
Assume that $\rank J(\bm{x})\neq 2$ for some nonzero $\bm{x}=(x_0:x_1:x_2:x_3)\in\overline{\qq}^4$ satisfying
$F_1(\bm{x})=G_1(\bm{x})=0$.
Then we have $x_i x_j=0$ for all $0\le i<j\le 3$.
From this and $F_1(\bm{x})=0$, we have $x_0=x_1=x_2=x_3=0$,
which is a contradiction.
Hence $\{F_1,G_1\}$ is a nonsingular system.
Similarly, $\{F_1,G_2\}$ is also a nonsingular system.

Therefore $Z(F_1,G_1)$ and $Z(F_1,G_2)$ are irreducible
curves of genus $1$ by Lemmas~\ref{lemma:BH} and \ref{lemma:genus_CI}.
\end{proof}

We are now in a position to complete the proof of Theorem~\ref{thm:parametrization0}.

\begin{proof}[Proof of Theorem~\ref{thm:parametrization0}]
By assumption, each irreducible component is parametrized by rational functions.
Hence each component is a curve of genus $0$ (see \cite[(2.16)]{Reid}).
By Lemma~\ref{lemma:curve_HK}, we have $(3a_2)^2=2(3a_4)$, that is, $3a_2^2=2a_4$.

As commented at the end of Subsection~\ref{subsect:QF}, the symmetric probability measures that correspond to the classical orthogonal polynomials are Hermite measure $e^{-t^2}dt/\sqrt{\pi}$ on $(-\infty,\infty)$ or Gegenbauer measure $(1-t^2)^{\lambda-1/2}dt/\int_{-1}^1 (1-t^2)^{\lambda-1/2} dt$ on $(-1,1)$.

When $w(t) = e^{-t^2}/\sqrt{\pi}$ on $(-\infty,\infty)$, by (\ref{eq:Hermite_moment}), the second and fourth moments are given by $a_2 = \frac{1}{2}$ and $a_4 = \frac{3}{4}$, which do not satisfy the condition that $3a_2^2=2a_4$.
When
$w(t) = (1-t^2)^{\lambda-1/2}/\int_{-1}^1 (1-t^2)^{\lambda-1/2} dt$ on $(-1,1)$, by (\ref{eq:moment_Gegenbauer2}) and (\ref{eq:moment_Gegenbauer3}), we have
\[
	a_2 = \frac{1}{2(\lambda + 1)}, \quad a_4 = \frac{3}{4(\lambda+1)(\lambda+2)}.
\]
Then
\[
    3 a_2^2 - 2 a_4
    = \frac{3((\lambda+2)-2(\lambda+1))}{4(\lambda+1)^2(\lambda+2)}
    = -\frac{3\lambda}{4(\lambda+1)^2(\lambda+2)}.
\]
Therefore the condition $3a_2^2=2a_4$ is equivalent to $\lambda=0$,
corresponding to the Chebyshev measure $(1-t^2)^{-1/2}dt/\pi$.
\end{proof}

\section{Proof of Theorem~\ref{thm:fivesix}} \label{sect:five-six}

For the Chebyshev measure $(1-t^2)^{-1/2}dt/\pi$, the equations (\ref{eq:HK4})
can be reduced to
\begin{equation}\label{eq:HK3}
 \begin{split}
	x_1^2 + x_2^2 + x_3^2 &= \frac{3}{2}, \\
	x_1^4 + x_2^4 + x_3^4 &= \frac{9}{8}.
 \end{split}
\end{equation}
The key part of the proof of Theorem~\ref{thm:fivesix} is Lemma~\ref{lem:5-6-1} below.
The essence of this lemma is the polynomial identity (\ref{eq:KawadaWooley}).

\begin{lemma} \label{lem:5-6-1}
Let $x_1,x_2,x_3 \in \mathbb{R}$ with $0 < x_1 < x_2 < x_3$. Let
\begin{align}
 x_1^2 + x_2^2 + x_3^2 &=\frac{3}{2}, \label{eq:5-6-1} \\
 x_1^4 + x_2^4 + x_3^4 &=\frac{9}{8}, \label{eq:5-6-2} \\
          x_1 + x_2 - x_3 & =0, \label{eq:5-6-3} \\
 x_1^2 + x_1x_2 + x_2^2& =\frac{3}{4}. \label{eq:5-6-4}
\end{align}
Then any two of the four statements being true imply the others.
\end{lemma}
 
\begin{proof}[Proof of Lemma~\ref{lem:5-6-1}]
Suppose (\ref{eq:5-6-1}) and (\ref{eq:5-6-2}). Then it follows that
\[
 2(x_1^4 + x_2^4 + x_3^4) = \frac{9}{4} = \Big( \frac{3}{2} \Big)^2 = (x_1^2 + x_2^2 + x_3^2)^2,
\]
which implies that
\begin{equation} \label{eq:5-6-5}
\begin{gathered}
 0 = x_3^4 - 2(x_1^2 + x_2^2)x_3^3 + (x_1 + x_2)^2(x_1-x_2)^2 \\ 
    = (x_3 - x_1 - x_2)(x_3 + x_1 + x_2)(x_3 + x_1 - x_2)(x_3 - x_1 + x_2).
\end{gathered}
\end{equation}
We obtain (\ref{eq:5-6-3}) since $0 < x_1 < x_2 < x_3$. Then substituting (\ref{eq:5-6-3}) into (\ref{eq:5-6-1}) gives
\begin{equation} \label{eq:5-6-6}
 \frac{3}{2} = x_1^2 + x_2^2 + x_3^2 = x_1^2 + x_2^2 + (x_1+x_2)^2
                = 2(x_1^2 + x_1x_2 + x_2^2),
\end{equation}
which implies (\ref{eq:5-6-4}). 

Suppose (\ref{eq:5-6-1}) and (\ref{eq:5-6-3}). Then we get (\ref{eq:5-6-4}) by the same argument as  in (\ref{eq:5-6-6}). By the polynomial identity (\ref{eq:KawadaWooley}), we have
\begin{equation} \label{eq:5-6-7}
 x_1^4+x_2^4+(x_1+x_2)^4 = 2(x_1^2 + x_1x_2 + x_2^2)^2 = \frac{9}{8}
\end{equation}
and thereby (\ref{eq:5-6-2}) from (\ref{eq:5-6-3}).

Suppose (\ref{eq:5-6-2}) and (\ref{eq:5-6-3}). As in (\ref{eq:5-6-7}), substituting (\ref{eq:5-6-3}) into (\ref{eq:5-6-2}) provides
\[
 \frac{9}{8} = x_1^4 + x_2^4 + (x_1+x_2)^4 = 2(x_1^2 + x_1x_2 + x_2^2)^2.
\]
Since $x_1^2+x_1x_2+x_2^2$ is positive semi-definite, we obtain (\ref{eq:5-6-4}). Then we obtain (\ref{eq:5-6-1}) by the argument as (\ref{eq:5-6-6}).
 
Suppose (\ref{eq:5-6-1}) and (\ref{eq:5-6-4}). Then
\[
 x_1^2+x_2^2+x_3^2 = \frac{3}{2} = 2(x_1^2 + x_1x_2 + x_2^2).
\]
Since
\[
 0 = 2(x_1^2 + x_1x_2 + x_2^2) - (x_1^2+x_2^2+x_3^2) = (x_1+x_2-x_3)(x_1+x_2+x_3)
\]
and $0 < x_1 < x_2 < x_3$, we obtain (\ref{eq:5-6-3}). Then (\ref{eq:5-6-1})  and (\ref{eq:5-6-3}) imply (\ref{eq:5-6-2}), as already shown above.
 
Suppose (\ref{eq:5-6-2}) and (\ref{eq:5-6-4}). Then
\[
 x_1^4+x_2^4+x_3^4 = \frac{9}{8} = 2(x_1^2 + x_1x_2 + x_2^2)^2.
\]
Since
\begin{eqnarray*}
	     0 &=& 2(x_1^2 + x_1x_2 + x_2^2)^2 - (x_1^4+x_2^4+x_3^4)  \\
            &=& \sum_{i=0}^4 \binom{4}{i} x_1^i x_2^{4-i} - x_3^4  \\
		   &=& (x_1+x_2)^4-x_3^4  \\
		   &=& \bigl((x_1+x_2)^2 + x_3^2\bigr) \bigl(x_1+x_2 + x_3\bigr) \bigl(x_1+x_2 - x_3\bigr),
\end{eqnarray*}
we obtain (\ref{eq:5-6-3}) since $0 < x_1 < x_2 < x_3$. Then (\ref{eq:5-6-2})  and (\ref{eq:5-6-3}) imply (\ref{eq:5-6-1}), as is already shown.
 
Finally, if (\ref{eq:5-6-3}) and (\ref{eq:5-6-4}) hold, then we obtain (\ref{eq:5-6-1}) by the same argument as in (\ref{eq:5-6-6}).  (\ref{eq:5-6-1}) and (\ref{eq:5-6-3}) then imply (\ref{eq:5-6-2}), as seen above.
\end{proof}
 
\begin{proof}[Proof of Theorem~\ref{thm:fivesix}]
By Lemma~\ref{lem:5-6-1} it suffices to show that the rational solutions of (\ref{eq:5-6-3}) and (\ref{eq:5-6-4}) are classified by (\ref{eq:fivesix1_0}).
As seen in Example~\ref{ex:Chebyshev1}, $(x_1,x_2,x_3) = (1/7, 11/14, -13/14)$ is a rational solution of (\ref{eq:5-6-3}) and (\ref{eq:5-6-4}).
Then every rational point on the ellipse $C: x_1^2 + x_1x_2 + x_2^2 = \frac{3}{4}$ is the intersection point of $C$ and the line $x_2 = t ( x_1 - \frac{1}{7} ) + \frac{11}{14}$ where $t \in \qq$. Combining this with (\ref{eq:5-6-4}), we have
\begin{equation}\label{eq:line5-6-1}
 x_1 = \frac{2t^2-22t-13}{14(t^2+t+1)}, \quad
 x_2 = \frac{-13t^2-4t+11}{14(t^2+t+1)}.
\end{equation}
Then it follows from (\ref{eq:5-6-3}) and (\ref{eq:line5-6-1}) that
\begin{equation}\label{eq:line5-6-2}
 x_3 = x_1+x_2 = - \frac{11t^2+26t+2}{14(t^2+t+1)}. \qedhere
\end{equation}
\end{proof}

\section{Proof of Theorem~\ref{thm:spectrum0}} \label{sect:degree5}

As seen in the previous sections, an antipodal $5$-design with $2N$ rational points for the  Chebyshev measure $(1-t^2)^{-1/2}dt/\pi$ on $(-1,1)$, is equivalent to a system of Diophantine equations of type
\begin{align}
x_1^2 + x_2^2 + \cdots  + x_N^2 &= \frac{N}{2}, \label{eq:chev_even:2} \\
x_1^4 + x_2^4 + \cdots  + x_N^4 &= \frac{3N}{8}. \label{eq:chev_even:4}
\end{align}
The proof of Theorem~\ref{thm:spectrum0} is separated into proofs of the two following  propositions:

\begin{proposition}\label{thm:spectrum_if}
There exists an antipodal $5$-design with $2N$ rational points for Chebyshev measure $(1-t^2)^{-1/2}dt/\pi$ on $(-1,1)$ if
\begin{equation}\label{eq:spectrum1}
    \begin{split}
        N=3k \ \text{ with } \ k \ge 1, \\
        N=3k+1 \ \text{ with } \ k \ge 6, \\
        N=3k+2 \ \text{ with } \ k \ge 3.
    \end{split}
\end{equation}
\end{proposition}

\begin{proposition}\label{thm:spectrum_onlyif}
There does not exist an antipodal $5$-design with $2N$ rational nodes for $(1-t^2)^{-1/2}dt/\pi$ if $N \in \{1,2,4,5,7,8,10,13,16\}$.
\end{proposition}

\subsection{Proof of Proposition~\ref{thm:spectrum_if}} \label{subsect:existence}

Our proof makes use of the following lemma.

\begin{lemma}
\label{lem:existence1}
Define three subsets $A, B, C$ of $\mathbb{Q}$ by
\begin{equation}
\label{eq:ABC}
\begin{gathered}
A = \Big\{ \pm \frac{2t^2-22t-13}{14(t^2+t+1)} \mid t \in \mathbb{Z} \Big\}, \\
B = \Big\{ \pm \frac{-13t^2-4t+11}{14(t^2+t+1)} \mid t \in \mathbb{Z} \Big\}, \\
C = \Big\{ \pm \frac{11t^2+26t+2}{14(t^2+t+1)} \mid t \in \mathbb{Z} \Big\}.
\end{gathered}
\end{equation}
Then
\[
|A\cap B| , |B\cap C| , |C\cap A| < \infty
\]
\end{lemma}
\begin{proof}[Proof of Lemma~{\rm \ref{lem:existence1}}]
We only prove $|A\cap B| < \infty$ since the proofs of the other cases are similar.
Since
\[
\lim_{|t|\to\infty} \frac{2t^2-22t-13}{14(t^2+t+1)} = \frac{1}{7}
\neq \mp \frac{13}{14} = \pm \lim_{|u|\to\infty} \frac{-13u^2-4u+11}{14(u^2+u+1)},
\]
there exists an integer $N$ such that if $|t|, |u| \ge N$ then
\[
\frac{2t^2-22t-13}{14(t^2+t+1)} \neq \pm\frac{-13u^2-4u+11}{14(u^2+u+1)}.
\]
Hence we have
\[
    A \cap B \subset
    \Big\{ \pm \frac{2t^2-22t-13}{14(t^2+t+1)} \mid t \in \mathbb{Z}, |t| < N \Big\}
    \cup \Big\{ \pm \frac{-13t^2-4t+11}{14(t^2+t+1)} \mid t \in \mathbb{Z}, |t| < N \Big\}.
\]
The right-hand side is a finite set, which implies $|A\cap B| < \infty$.
\end{proof}

\begin{proof}[Proof of Proposition~\ref{thm:spectrum_if}]
By Lemma~\ref{lem:existence1} there exists some positive integer $t_0$ such that for all integers $t_1, t_2, t_3 \ge t_0$, 
\begin{equation}
\label{eq:explicit0}
\alpha_{t_1} := \frac{2t_1^2-22t_1-13}{14(t_1^2+t_1+1)}, \quad
\beta_{t_2} := \frac{-13t_2^2-4t_2+11}{14(t_2^2+t_2+1)}, \quad
\gamma_{t_3} := \frac{11t_3^2+26t_3+2}{14(t_3^2+t_3+1)}
\end{equation}
and
their antipodal pairs $-\alpha_{t_1}, -\beta_{t_2}, -\gamma_{t_3}$ are mutually distinct.
Proposition~\ref{prop:composition} then implies that for any positive integer $k$
\begin{equation}
\label{eq:explicit1}
\begin{gathered}
\frac{1}{\pi} \int_{-1}^1 \frac{f(t)}{\sqrt{1-t^2}}dt \\
= \frac{1}{6k} \sum_{t=t_0}^{t_0+k-1} (f(\alpha_t) + f(\beta_t) + f(\gamma_t) + f(-\alpha_t) + f(-\beta_t) + f(-\gamma_t))
\end{gathered}
\end{equation}
is an antipodal $5$-design with $6k$ distinct rational nodes.

Next, we note that the antipodal pairs of the point configuration
\begin{equation}
\label{eq:22points-1}
(x_1, \ldots, x_{11}) = \frac{1}{90} (2,8,16,34,72,73,76,77,80,84,86)
\end{equation}
is an antipodal $5$-design with $22$ rational points.
By Lemma~\ref{lem:existence1}, there exists $t_0 \in \mathbb{Z}$ such that for all integers $t_1, t_2, t_3 \ge t_0$
\[
\pm \frac{2t_1^2-22t_1-13}{14(t_1^2+t_1+1)},
\; \pm \frac{-13t_2^2-4t_2+11}{14(t_2^2+t_2+1)},
\; \pm \frac{11t_3^2+26t_3+2}{14(t_3^2+t_3+1)}
\]
and $\pm x_i$ ($i=1,\ldots,11$) are mutually distinct. Then by Proposition~\ref{prop:composition}, we get an antipodal $5$-design with $6k+4$ distinct rational nodes for $k \ge 3$.

Finally, by taking the antipodal pair of  the point configuration
\begin{equation}
\label{eq:38points-1}
\begin{gathered}
(x_1, \ldots, x_{19}) \\
= \frac{1}{126}
(2, 4, 20, 32, 40, 44, 56, 83, 88, 100, 104, 106, 109, 110, 116, 118, 120, 122, 124),
\end{gathered}
\end{equation}
we obtain an antipodal $5$-design with $38$ rational points.
Again by combining Lemma~\ref{lem:existence1} and Proposition~\ref{prop:composition}, we obtain an antipodal $5$-design with $6k+2$ distinct rational nodes for $k \ge 6$.
\end{proof}

\begin{remark}\label{rem:computational} 
The derivation of the set (\ref{eq:38points-1}) depends on a computer search based on a divide-and-conquer algorithm.
\end{remark}

\subsection{Proof of Proposition~\ref{thm:spectrum_onlyif}} \label{subsect:nonexistence}

Our proof of Proposition~\ref{thm:spectrum_onlyif} makes use of elementary congruence conditions.
Despite the simplicity of our methods, the arguments below will work for a variety of probability measures including Hermite measure $e^{-t^2}dt/\sqrt{\pi}$ on $(-\infty,\infty)$; see Section~\ref{sect:conclusion} for the detail.

Below we consider a system of two Diophantine equations of type
\begin{align}
x_1^2 + x_2^2 + \cdots  + x_N^2 &= \frac{A}{2^\alpha C}, \label{eq:KH2} \\
x_1^4 + x_2^4 + \cdots  + x_N^4 &= \frac{B}{2^\beta D},  \label{eq:KH4}
\end{align}
where $\alpha \le 2$ (not necessarily positive), $1 \le \beta \le 4$ are integers, $A$, $B$, $C$, $D$ are odd integers, and $\gcd(A, C) = \gcd(B, D)  =1$.

\begin{example}\label{ex:spectrum1}
For example, (\ref{eq:KH2}) and (\ref{eq:KH4}) with $(A,B,C,D) = (3,9,1,1)$ and $(\alpha,\beta)=(1,3)$ are just two equations given in (\ref{eq:HK3}).
\end{example}

The following is a collection of well-known facts in elementary number theory. 
 
\begin{lemma}\label{lem:pwr_mod4}
Let $n$ be a positive integer. Then the following hold:
\begin{enumerate}
\item[(i)] 
	$n^2 \equiv 0 \pmod{4}$ if $n$ is even, and $n^2 \equiv 1 \pmod{4}$ if $n$ is odd;
\item[(ii)] 
	$n^4 \equiv 0 \pmod{16}$ if $n$ is even, and $n^4 \equiv 1 \pmod{16}$ if $n$ is odd;
\item[(iii)] 
	$n^2 \equiv n^4 \pmod{4}$ regardless of the parity of $n$. 
\end{enumerate}
\end{lemma}

The following two lemmas are key in the proof of Proposition~\ref{thm:spectrum_onlyif}.

\begin{lemma}\label{lem:nonexistR}
There does not exist a rational solution of \eqref{eq:KH4} 
for any integer $1 \le N \le 15$, 
if  $1 \le \beta \le 4$ and $N < R$, where $R$ is the smallest positive integer such that 
$DR \equiv 2^{4-\beta}  B \pmod{16}$.
\end{lemma}
\begin{proof}[Proof of Lemma~\ref{lem:nonexistR}]
Suppose that $\{x_1, \ldots, x_N\}$ is a rational solution of \eqref{eq:KH4}. 
We write $x_i = m_i/k$, where $m_1, \ldots, m_N, k$ are positive integers and 
$\gcd(m_1, \ldots, m_N, k) = 1$. 
Then \eqref{eq:KH4} is equivalent to the equation
\begin{equation}\label{eq:nonexistR:1}
2^\beta D \sum_{i=1}^{N} m_i^4 = B k^4.
\end{equation}
Since $\beta \ge 1$,  we have $B k^4 \equiv 0 \pmod{2}$ and hence $k$ is even, say $k = 2h$. Replacing $x_i$ by $m_i/(2h)$ in \eqref{eq:KH4}, we have
\begin{align}
D \sum_{i=1}^{N} m_i^4 &= 2^{4-\beta} B h^4. \label{eq:nonexistR:2}
\end{align}

We denote by $T$ the number of odd integers among $m_1, \ldots, m_N$, that is 
\begin{equation}\label{eq:odd}
T = \abs{ \{  1 \le i \le N \mid m_i \equiv 1 \pmod{2} \} }.
\end{equation}
It may not be entirely obvious, but it can be shown that $T \ne 0$. Hence we have $1 \le T \le N$.

By reducing \eqref{eq:nonexistR:2} modulo $16$ and then using Lemma~\ref{lem:pwr_mod4}~(ii), we have 
\begin{equation}\label{eq:nonexistR:3}
DT \equiv 2^{4-\beta} B h^4 \pmod{16}.
\end{equation}
If $h$ is even, the congruence \eqref{eq:nonexistR:3} implies $T \equiv 0 \pmod{16}$, which contradicts $1 \le T \le N \le 15$. Hence, $h$ must be odd. 
Since $h^4 \equiv 1 \pmod{16}$, we have $DT \equiv 2^{4-\beta}B \pmod{16}$.
Therefore we have $T \ge R$, where $R$ is the smallest positive integer such that $DR \equiv 2^{4-\beta}B \pmod{16}$. 
\end{proof}

\begin{lemma}\label{lem:nonexist}
There does not exist a rational solution of \eqref{eq:KH2} and \eqref{eq:KH4} 
for any integer $1 \le N \le 15$ if one of the following conditions holds:
\begin{enumerate}
\item[(i)] $\alpha \le 0$ and $\beta \in \{3, 4\}$;
\item[(ii)] $(\alpha, \beta) \in \{1\} \times \{ 1, 2, 4 \} \cup \{2\} \times \{ 1, 2, 3 \}$;
\item[(iii)] $(\alpha, \beta)=(2,4)$ and $AD \not\equiv BC \pmod{4}$.
\end{enumerate}
\end{lemma}
\begin{proof}[Proof of Lemma~\ref{lem:nonexist}]
Suppose $\{x_1, \ldots, x_N\}$ is a rational solution of \eqref{eq:KH2} and \eqref{eq:KH4}. 
We write $x_i = m_i/k$, where $m_1, \ldots, m_N, k$ are positive integers and 
$\gcd(m_1, \ldots, m_N, k) = 1$. 
Proceeding as in the proof of Lemma~\ref{lem:nonexistR}, we can show that $k$ is even, say $k = 2h$.
Replacing $x_i$ by $m_i/(2h)$ in \eqref{eq:KH2} and \eqref{eq:KH4}, we have
\begin{align}
C \sum_{i=1}^{N} m_i^2 &= 2^{2-\alpha} A h^2, \label{eq:KH2_dioph}\\
D \sum_{i=1}^{N} m_i^4 &= 2^{4-\beta} B h^4. \label{eq:KH4_dioph}
\end{align}

As with (\ref{eq:odd}), let $T$ be the number of odd integers among $m_1, \ldots, m_N$.
As discussed in the proof of Lemma~\ref{lem:nonexistR}, we have $1 \le T \le N$. 

Reducing \eqref{eq:KH2_dioph} and \eqref{eq:KH4_dioph} modulo $4$ and $16$, respectively, 
it follows from Lemma~\ref{lem:pwr_mod4}~(i) and (ii)
that
\begin{align}
CT &\equiv 2^{2-\alpha} A h^2 \pmod{4}, \label{eq:KH2_dioph:mod4} \\
D T &\equiv 2^{4-\beta} B h^4 \pmod{16}. \label{eq:KH4_dioph:mod16}
\end{align}

Suppose that $h$ is even.
Then it follows from \eqref{eq:KH4_dioph:mod16} that $T \equiv 0 \pmod{16}$.
However, since $T \ne 0$, we have $N \ge T \ge 16$. 

Suppose that $h$ is odd. 
It follows from Lemma~\ref{lem:pwr_mod4}~(iii) that $h^2 \equiv h^4 \equiv 1 \pmod{4}$. 
Multiplying $\eqref{eq:KH2_dioph:mod4}$ and $\eqref{eq:KH4_dioph:mod16}$ by $D$ and $C$ respectively, we have
$CDT \equiv 2^{2-\alpha} AD \equiv 2^{4-\beta} BC \pmod{4}$,  
namely
\begin{align*}
0 &\equiv 2^{4-\beta} BC \pmod{4}, &&\text{if }   \alpha \le 0, \\
2 AD&\equiv 2^{4-\beta} BC \pmod{4}, &&\text{if }   \alpha = 1, \\
AD&\equiv 2^{4-\beta} BC \pmod{4}, &&\text{if }   \alpha = 2.
\end{align*}
Since both $AD$ and $BC$ are odd, 
the above congruences lead to a contradiction if one of the following cases occurs:
\begin{enumerate}
\item $\alpha \le 0$ and $\beta \in \{3, 4\}$, 
\item $\alpha = 1$ and $\beta \in \{1, 2, 4\}$, 
\item $\alpha = 2$ and $\beta \in \{1, 2, 3\}$,
\item $(\alpha, \beta) = (2, 4)$ and $AD \not\equiv BC \pmod{4}$.
\end{enumerate}

In summary,  for any $N \le 15$, under the conditions given in the statement, 
the equations \eqref{eq:KH2_dioph} and \eqref{eq:KH4_dioph} have no integer solutions.
\end{proof}

We are ready to complete the proof of Proposition~\ref{thm:spectrum_onlyif}.

\begin{proof}[Proof of Proposition~\ref{thm:spectrum_onlyif}]
Let $N \in \{1, 2, 4, 5, 7, 10, 13\}$. Then by applying Lemma~\ref{lem:nonexistR} with $\alpha,\beta,R$ listed in Table~\ref{tab:value_R} below, we find that there does not exist a rational solution of the equations \eqref{eq:chev_even:2} and \eqref{eq:chev_even:4}.

Thus it remains to consider the case $N = 8,16$.
Suppose that there exists a rational solution of \eqref{eq:chev_even:2} and \eqref{eq:chev_even:4}.
We write $x_i=m_i/k$, where $m_1,\dotsc,m_N$ and $k$ are integers and $\gcd(m_1,\dotsc,m_N,k)=1$.
Let $N=8N'$. Then we have
\begin{align}
	m_1^2 + \dotsb + m_N^2 &= 4 k^2 N', \label{eq:quadratic3} \\
	m_1^4 + \dotsb + m_N^4 &= 3 k^4 N'. \label{eq:quartic3}
\end{align}
Since $m_i^2\equiv m_i^4\pmod{4}$, we have $4 k^2 N'\equiv 3 k^4 N' \pmod{4}$ and therefore $k^4 N'\equiv 0\pmod{4}$.
Since $N'\in\{1,2\}$, $k$ is even.
Then we have
\begin{equation}\label{eq:quartic4}
	m_1^4 + \dotsb + m_N^4 \equiv 0 \pmod{16}.
\end{equation}

As with (\ref{eq:odd}), let $T$ be the number of odd integers among $m_1, \dotsc, m_N$.
Since $m_1^4 + \dotsb + m_N^4 \equiv T \pmod{16}$,
by \eqref{eq:quartic4}, we have $T\equiv 0\pmod{16}$.
Since $\gcd(m_1, \dotsc, m_N, k)=1$ and $k$ is even, we have $1 \le T \le N$.
When $N=8$, this is a contradiction.

When $N=16$, we have $T=16$.
Note that $r^2 \equiv 1, 9, 17, 25 \pmod{32}$ for an odd integer $r$.
When $r^2 \equiv 1, 9, 17, 25 \pmod{32}$,
we have $r^4 \equiv 1, 17, 1, 17 \pmod{32}$, respectively.
For $j \in \{1, 9, 17, 25\}$, let
\[
	T_j = \lvert \{ 1 \le i \le N \mid m_i^2 \equiv j \pmod{32} \} \rvert.
\]
Since $k$ is even and $N'=2$, by \eqref{eq:quadratic3} and \eqref{eq:quartic3}, we have
\begin{align*}
	T_1 + 9 T_9 + 17 T_{17} + 25 T_{25} &\equiv 0 \pmod{32}, \\
	T_1 + 17 T_9 + T_{17} + 17 T_{25} &\equiv 0 \pmod{32}.
\end{align*}
Since $T_1 + T_9 + T_{17} + T_{25} = N = 16$, we have
\begin{align}
	8 T_9 + 16 T_{17} + 24 T_{25} + 16 &\equiv 0 \pmod{32}, \label{eq:T_j-1} \\
	16 T_9 + 16 T_{25} + 16 &\equiv 0 \pmod{32}. \label{eq:T_j-2}
\end{align}
Dividing \eqref{eq:T_j-1} and \eqref{eq:T_j-2} by $8$ and $16$, respectively, we have
\begin{align}
	T_9 + 2 T_{17} + 3 T_{25} + 2 &\equiv 0 \pmod{4}, \label{eq:T_j-3} \\
	T_9 + T_{25} + 1 &\equiv 0 \pmod{2}. \label{eq:T_j-4}
\end{align}
By \eqref{eq:T_j-4}, there exists an integer $S$ such that $T_9+T_{25}+1=2S$.
Therefore, by \eqref{eq:T_j-3}, we obtain
\[
	2 T_{17} + 2 T_{25} + 2 S + 1 \equiv 0 \pmod{4}.
\]
This is a contradiction.
\end{proof}

\begin{table}[ht!]
\caption{Values of $\alpha,\beta,R$ used in the proof of Proposition~\ref{thm:spectrum_onlyif}}
\label{tab:value_R}
\begin{tabular}{|l|c|c|c|c|c|c|c|}
\hline
$N$ &  $1$ & $2$ &  {4} & $5$  & $7$  & $10$ & $13$ \\ \hline
$(\alpha, \beta)$ &  $(1,3)$ & $(0, 2)$ & {$(-1,1)$} & $(1,3)$  & $(1,3)$ & $(0, 2)$ & $(1,3)$ \\ \hline
$R := 6N \mod 16$ & $6$  & $12$ & {$8$} & $14$ & $10$ & $12$ & $14$ \\ \hline
\end{tabular}
\end{table}

\section{Proof of Theorem~\ref{thm:PTEparametric}} \label{sect:PTE}

We start with a simple but significant observation that creates a novel connection between the PTE problem and rational interval designs.

\begin{proposition} \label{prop:PTE1}
Assume there exist two distinct rational $m$-designs for $\int_I w(t)dt$. Then there exists a pair of disjoint multisets of integers, say $X$ and $Y$, such that $(X,Y)$ is a solution of the PTE problem of degree $m$.
\end{proposition}
\begin{proof}[Proof of Proposition~{\rm \ref{prop:PTE1}}]
Let $\tilde{X}$ and $\tilde{Y}$ be $m$-designs. Then it follows that
\[
 \frac{1}{n} \sum_{\tilde{x} \in \tilde{X}} \tilde{x}^k = \int_I t^k w(t)dt = \frac{1}{n} \sum_{\tilde{y} \in \tilde{Y}} \tilde{y}_i^k, \quad k=1,\ldots,m.
\]
Set $X = \tilde{X} \setminus (\tilde{X} \cap \tilde{Y})$ and $Y = \tilde{Y} \setminus (\tilde{X} \cap \tilde{Y})$.
Then $X$ and $Y$ satisfy the power-sum condition (\ref{eq:PTE1}).
\end{proof}

Before going to the proof of Theorem~\ref{thm:PTEparametric}, we make clear the definition of affine equivalence.

\begin{definition} [Affine equivalence] \label{def:equiv}
Two solutions $(\{x_1,\dotsc,x_N\},\{y_1,\dotsc,y_N\})$ and $(\{x'_1,\dotsc,x'_N\},\{y'_1,\dotsc,y'_N\})$ are \emph{affinely equivalent over $\qq$} if there exist $A,B\in\qq$ such that $\{Ax_i+B\}=\{x'_i\}$ and $\{Ay_i+B\}=\{y'_i\}$ as multisets.
\end{definition}
\begin{proof}[Proof of Theorem~\ref{thm:PTEparametric}]
The former part of the theorem follows from Theorem~\ref{thm:fivesix} and Proposition~\ref{prop:PTE1}.

To prove the latter part, we first define $(X,Y)$ by
\begin{equation}\label{eq:Borwein6}
\begin{gathered}
	x_1 = 2 n + 2 m, \quad
	x_2 = -n m - n - m + 3, \quad
	x_3 = n m - n - m - 3, \\
	y_1 = 2 n - 2 m, \quad
	y_2 = - n + n m + m + 3, \quad
	y_3 = m - n m - n - 3, \\
    x_{3+i} = -x_i, \; y_{3+i} = -y_i \quad \text{ for $i = 1,2,3$}
\end{gathered}
\end{equation}
where $m,n \in \mathbb{Q}$.
In comparison with (\ref{eq:Borwein6}),  the signs of $x_2$ and $y_2$ are changed so that the solution is \emph{linear} in the sense that
\[
x_1+x_2+x_3=y_1+y_2+y_3=0.
\]
Let $(X',Y')$ be the dehomogenization of the solution in (\ref{eq:PTE2}),
defined by
\begin{equation}
\begin{gathered}
	x'_1 = \frac{2s^2-22s-13}{14(s^2+s+1)}, \;
	x'_2 = \frac{-13s^2-4s+11}{14(s^2+s+1)}, \;
	x'_3 = \frac{11s^2+26s+2}{14(s^2+s+1)}, \\
	y'_1 = \frac{2t^2-22t-13}{14(t^2+t+1)}, \;
	y'_2 = \frac{-13t^2-4t+11}{14(t^2+t+1)}, \;
	y'_3 = \frac{11t^2+26t+2}{14(t^2+t+1)}, \\
    x'_{3+i}=-x'_i, \; y'_{3+i}=-y'_i \quad \text{ for $i=1,2,3$}.
\end{gathered}
\end{equation}
Assume $(X,Y)$ and $(X',Y')$ are affinely equivalent over $\qq$.
Then there exist $A,B\in\qq$ such that $\{Ax_i+B\}=\{x'_i\}$ and $\{Ay_i+B\}=\{y'_i\}$ as multisets.
Hence we have
\[
    \sum_{i=1}^6 (Ax_i+B) = \sum_{i=1}^6 x'_i.
\]
Since $(X,Y)$ and $(X',Y')$ are antipodal, we have $B=0$.

Now we consider the equations $Ax_1=x'_1$, $Ax_2=x'_2$ and $Ay_1=y'_1$.
Note that $Ax_1=x'_1$ and $Ax_2=x'_2$ imply $Ax_3=x'_3$ since both the solutions are linear.
By \eqref{eq:Borwein6}, we have
\begin{gather}
    2 A (m + n) = x'_1, \label{eq:equiv1} \\
    A (-n m - n - m + 3) = x'_2, \label{eq:equiv2} \\
    2 A (n - m) = y'_1. \label{eq:equiv3}
\end{gather}
By \eqref{eq:equiv1} and \eqref{eq:equiv3}, we have
\begin{equation}\label{eq:mn}
    m = \frac{x'_1 - y'_1}{4A}, \quad
    n = \frac{x'_1 + y'_1}{4A}.
\end{equation}
Substituting \eqref{eq:mn} into \eqref{eq:equiv2} gives
\begin{equation}\label{eq:quadratic_A}
    48 A^2 - 8 (x'_1 + 2 x'_2) A - (x'_1)^2 + (y'_1)^2 = 0.
\end{equation}
Then the discriminant of the left-hand side of \eqref{eq:quadratic_A}
is equal to
\[
    64 (4((x'_1)^2 + x'_1 x'_2 + (x'_2)^2) - 3 (y'_1)^2)
    = \left(\frac{12(2t+3)(4t-1)}{7(t^2+t+1)}\right)^2.
\]
We note that
\[
4((x'_1)^2 + x'_1 x'_2 + (x'_2)^2)=3
\]
as in the proof of
Theorem~\ref{thm:fivesix} (see Lemma~\ref{lem:5-6-1}).
By solving \eqref{eq:quadratic_A} and using \eqref{eq:mn}, we have
\begin{equation}\label{eq:Amn}
\begin{aligned}
    &(A,m,n) \\
    &= \left(-\frac{(2 s t + t + s + 2)(8 s t + 5 t + 5 s - 3)}{56 (s^2 + s + 1)(t^2 + t + 1)},
    \frac{3(t - s)}{2 s t + t + s + 2},
    -\frac{2 s t - 11 s - 11 t - 13}{8 s t + 5 t + 5 s - 3}\right), \\
    &\quad \left(\frac{(s - t)(2 s t - 11 s - 11 t - 13)}{56 (s^2 + s + 1)(t^2 + t + 1)},
    \frac{3(8 s t + 5 t + 5 s - 3)}{2 s t - 11 s - 11 t - 13},
    \frac{2 s t + s + t + 2}{s - t}
    \right).
\end{aligned}
\end{equation}
Substituting \eqref{eq:Amn} into \eqref{eq:Borwein6}, we have
\[
    A y_2 = \frac{-13 t^2 - 4 t + 11}{14(t^2 + t + 1)}, \quad
    \frac{11 t^2 + 26 t + 2}{14(t^2 + t + 1)},
\]
respectively.
The former means $Ay_2=y'_2$, which implies $Ay_3=y'_3$.
Similarly, the latter implies $Ay_2=y'_3$ and $Ay_3=y'_2$.
Therefore the solutions in (\ref{eq:PTE2}) are included in Borwein's solutions up to equivalence over $\qq$.
\end{proof}

The following observation is due to Hideki Matsumura.

\begin{remark}\label{rem:Chernick1}
Another parametric solution for $(M,N)=(5,6)$ was found by Chernick~\cite[pp.629-630]{C1937} as follows:
\begin{equation}\label{eq:Chernick1}
\begin{gathered}
	\alpha_1 = -5m^2+4mn-3n^2, \quad
	\alpha_2 = -3m^2+6mn+5n^2, \quad
	\alpha_3 = -m^2-10mn-n^2, \\
	\beta_1 =  -5m^2+6mn+3n^2, \quad
	\beta_2 = -3m^2-4mn-5n^2, \quad
	\beta_3 = -m^2+10mn-n^2, \\
\alpha_{3+i} = -\alpha_i, \beta_{3+i} = -\beta_i \quad \text{for $i = 1,2,3$}.
\end{gathered}
\end{equation}
Our solution (\ref{eq:PTE2}) and Chernick's one are not affinely equivalent over $\qq$.
Suppose contrary.
Then there exists $A \in \qq$ such that
\[
14A^2(5m^2+2mn+n^2)(5m^2-2mn+n^2) = 3.
\]
Clearly, we have $n \ne 0$.
Let $x = m/n$ and $C$ be the smooth curve of genus $1$ defined by
\[
C: y^2 = 42(5x^2+2x+1)(x^2-2x+5).
\]
We can verify that $C(\qq_3) = \emptyset$, where $\qq_3$ is the field of $3$-adic numbers.
Therefore we conclude that $C(\qq) = \emptyset$.
\end{remark}

\section{Concluding remarks and future works}
\label{sect:conclusion}

The Hilbert-Kamke equations \eqref{eq:HK0} can be dealt with the circle method,
or the Hardy-Littlewood method.
We refer the reader to~\cite{Wooley2025} for recent developments on Hilbert-Kamke problem involving the circle method.

In the case of antipodal solutions of the Hilbert-Kamke equations of degree $5$, we need to solve the system of equations
\begin{equation}\label{eq:HK6}
    \begin{aligned}
        x_1^2 + x_2^2 + \cdots  + x_N^2 &= N a_2, \\
        x_1^4 + x_2^4 + \cdots  + x_N^4 &= N a_4.
    \end{aligned}
\end{equation}
We can show that \eqref{eq:HK6} has a nonsingular real solution if and only if $N a_2^2 > a_4 > a_2^2$.
Note that the inequality $a_4 > a_2^2$ holds if $a_2$ and $a_4$ are the moments of
a symmetric probability measure.

More generally, we consider the system of homogeneous equations
\begin{equation}\label{eq:diagonal}
    \begin{aligned}
        c_1 x_1^k + c_2 x_2^k + \cdots  + c_s x_s^k &= 0, \\
        d_1 x_1^n + d_2 x_2^n + \cdots  + d_s x_s^n &= 0,
    \end{aligned}
\end{equation}
where $k>n\ge 1$ and $c_1,c_2,\ldots,c_s,d_1,d_2,\ldots,d_s$ are integers with $c_i d_i\neq 0$ for $1\le i\le s$.
Let $\Gamma^*(k,n)$ be the least integer $r$ such that, whenever $s\ge r$, the system \eqref{eq:diagonal} has a nontrivial $p$-adic solution for every prime $p$.
By the results of Davenport-Lewis~\cite[Theorem~1]{DL1963} and Leep-Schmidt~\cite[(2.11)]{LS1983},
we have $\Gamma^*(k,n)\le(k^2+1)(n^2+1)$ and hence $\Gamma^*(4,2)\le 85$.
Let $G^*(k,n)$ be the least integer $r$ such that, whenever $s\ge r$, the system \eqref{eq:diagonal} has a nontrivial integral solution if it has a nonsingular real solution and a nonsingular $p$-adic solution for every prime $p$.
Parsell~\cite[Theorem~1.1]{Parsell2002} obtained the values of $G^*(k,n)$ for small $k$ and $n$.
In particular, we have $G^*(4,2)\le 20$.
Combining these results, we infer that the system \eqref{eq:HK6} has a rational solution if $N\ge\max\{84,a_4/a_2^2\}$ and if \eqref{eq:HK6} defines a nonsingular variety.
Since $a_2=1/2$ and $a_4=3/8$ for Chebyshev measure $(1-t^2)^{-1/2}dt/\pi$, the above condition on $N$ becomes $N\ge 84$.
Theorem~\ref{thm:spectrum0} is sharper than the above bound obtained by the circle method.
It is our hope that the design-theoretic approaches used in the present paper could be applied to the study of Hilbert-Kamke equations for other classical measures.

\begin{center}
* * *
\end{center}

A natural question asks whether there exists a rational $5$-design with $2N+1$ points for Chebyshev measure $(1-t^2)^{-1/2}dt/\pi$, namely
\begin{align}
\frac{1}{2N+1}\left( f(0) + \sum_{i=1}^{N} \big( f(x_i) + f(-x_i) \big) \right)
&= \frac{1}{\pi} \int_{-1}^{1} \frac{f(t)}{\sqrt{1-t^2}} dt
\quad \text{ for every $f \in \calP_5(\rr)$}. \label{eq:chev_odd}
\end{align}
Such design is equivalent to a disjoint solution of the Diophantine equations
\begin{equation} \label{eq:chev_odd:24} \begin{gathered}
x_1^2 + x_2^2 + \cdots  + x_N^2 = \frac{2N+1}{4}, \\
x_1^4 + x_2^4 + \cdots  + x_N^4 = \frac{3(2N+1)}{16}.
\end{gathered} \end{equation}
By applying Lemma~\ref{lem:nonexist} with $(\alpha, \beta) = (2, 4)$ to \eqref{eq:KH2} and \eqref{eq:KH4}, we establish that for every $1 \le N \le 15$ there does not exist a rational solution of \eqref{eq:chev_odd:24}.
Moreover the nonexistence of antipodal rational $5$-designs with $33$ points can be proved in a similar manner to the proof of the nonexistence of $5$-designs with $32$ points.
In summary, we obtain the following result.

\begin{theorem}\label{thm:nonexist_chev_odd}
If $1 \le N \le 16$, then there do not exist antipodal $5$-designs with $2N+1$ rational points for $(1-t^2)^{-1/2}dt/\pi$.
\end{theorem}

Meanwhile, a $5$-design with $35$ rational points actually exists. Such an example is obtained by taking the antipodal pair of the point configuration
\begin{equation}
\label{eq:35points-1}
\begin{gathered}
(x_1, \ldots, x_{17}) \\
= \frac{1}{1092}(9,65,91,195,531,669,689,729,837,871,923,933,1001,1027,1053,1066,1079)
\end{gathered}
\end{equation}
plus the origin $0$. We thus obtain the following result by the same way as in the proof of Theorem~\ref{thm:spectrum0}.

\begin{theorem}\label{thm:spectrum_odd}
Let $N$ be a positive integer such that $N \equiv 5 \pmod{6}$ and $N \notin \{5,11\}$. Then there exists an antipodal $5$-design with $2N+1$ rational points. 
\end{theorem}

\begin{problem}\label{prob:spectrum_odd}
Find an antipodal $5$-design with $37$ rational points, as the smallest open case.
\end{problem}

\begin{center}
* * *
\end{center}

Another question is a problem to find a rational $t$-designs for Chebyshev measure $(1-t^2)^{-1/2}dt/\pi$ and for larger $t$.
Here we only consider the case $t=7$.
An antipodal rational $7$-design with $2N$ points
\[
\frac{1}{2N} \sum_{i=1}^{N} \big( f(x_i) + f(-x_i) \big)
= \frac{1}{\pi} \int_{-1}^{1} \frac{f(t)}{\sqrt{1-t^2}} dt
\quad \text{ for every $f \in \calP_7(\rr)$}
\]
is equivalent to a disjoint solution of the Diophantine equations
\[
\begin{gathered}
x_1^2 + \cdots + x_N^2 = \frac{N}{2}, \\
x_1^4 + \cdots + x_N^4 = \frac{3N}{8}, \\
x_1^6 + \cdots + x_N^6 = \frac{5N}{16}.
\end{gathered}
\]
Note that if there do not exist antipodal $5$-designs with $n$ rational points, then there do not exist antipodal $7$-designs with $n$ rational points.
Then we can prove the following theorem by an argument similar to the proof of Proposition~\ref{thm:spectrum_onlyif}. 

\begin{theorem}\label{thm:nonexist_chev7}
If $1 \le N \le 16$, then there do not exist antipodal $7$-designs with $2N$ rational points for $(1-t^2)^{-1/2}dt/\pi$.
\end{theorem}

\begin{problem}
Determine whether there exists an antipodal $7$-design with $2N$ rational points for $n\ge 17$ and obtain an analogue of Theorem~\ref{thm:spectrum0}.
\end{problem}

\begin{center}
* * *
\end{center}

A challenging problem is to establish an analogue of Theorem~\ref{thm:spectrum0} for symmetric classical measures other than $(1-t^2)^{-1/2}dt/\pi$.
A significant case to be handled will be the Hermite measure $e^{-t^2}dt/\sqrt{\pi}$. A rational design for $e^{-t^2}dt/\sqrt{\pi}$ of type
\[
\frac{1}{2N} \sum_{i=1}^N (f(x_i) + f(-x_i)) = \frac{1}{\sqrt{\pi}} \int_{-\infty}^\infty f(t) e^{-t^2}dt
\quad \text{ for every $f \in \mathcal{P}_5(\rr)$},
\]
is equivalent to a system of Diophantine equations of type
\[
\begin{gathered}
x_1^2 + \cdots + x_N^2 = \frac{N}{2}, \\
x_1^4 + \cdots + x_N^4 = \frac{3N}{4}.
\end{gathered}
\]
Then Lemmas~\ref{lem:nonexistR} and~\ref{lem:nonexist} can be applied in the proof of the following result.

\begin{theorem}\label{thm:nonexist_gauss}
For any positive integer $1 \le N \le 15$, except for $N=14$, there do not exist antipodal $5$-designs with $2N$ rational points for $e^{-t^2}dt/\sqrt{\pi}$ on $(-\infty,\infty)$.
\end{theorem}

The exceptional case $N=14$ is suggestive.
In fact, by taking the antipodal pairs of the point configuration
\begin{equation}\label{eq:Hermite2}
(x_1,\ldots,x_{14}) = \frac{1}{30} (1,2,3,4,5,8,9,13,14,15,25,39,40,42),
\end{equation}
we obtain an antipodal $5$-design with $28$ rational points.
However, the following problem is beyond the grasp of our methods used in the present paper, to which we intend to return elsewhere.

\begin{problem}\label{prob:series_gauss}
Find an infinite family of rational $5$-designs with $2N$ points for Hermite measure $e^{-t^2}dt/\sqrt{\pi}$ and then establish an analogue of Theorem~\ref{thm:spectrum0}.
\end{problem}

It is also a challenging problem to obtain an analogue of Theorem~\ref{thm:spectrum0} for Gegenbauer measure other than Chebyshev measure $(1-t^2)^{-1/2}dt/\pi$, and we leave it for a future work.

\section*{Acknowledgement}
This work started about 10 years ago when the third author first got contact with Koichi Kawada to inquire about the significance of Hilbert identity in the modern study of Waring problem.
The authors gradually came to realize the validity of the polynomial identity approach for the degree four and five cases, and finally ended up writing the present paper.
We would like to express our sincere appreciation to him.
It is also our great pleasure to thank Eiichi Bannai who kindly told us~\cite[p.208, Problem 2]{BBI2016} and then motivated us to study the explicit construction of rational interval designs.
The authors would also like to thank Hideki Matsumura for kindly informing us an interesting observation on Chernick's solution of the PTE problem (Remark~\ref{rem:Chernick1}).
The third author would like to thank Rossman Wayne for giving many valuable suggestions to improve the presentation of this paper.
The authors would like to thank the referee for helpful comments and suggestions.


\end{document}